\documentclass[reqno,11pt]{amsart}
\usepackage{tkz-euclide} 
\usepackage{hyperref}
\usepackage{amsmath}  
\usepackage{amsfonts}
\usepackage{amssymb} 
\usepackage{mathrsfs} 
\usepackage{graphicx}  
\usepackage{bm}
\usepackage{tikz} 
\usetikzlibrary{svg.path} 

\usetikzlibrary{calc} 
\usepackage{xcolor}  
\usetikzlibrary{patterns}
\usetikzlibrary{arrows.meta} 
\usepackage{amsthm} 
\usepackage{float}
\usepackage{enumitem}
\usepackage[english]{babel}
\usepackage[font=footnotesize, labelfont=bf]{ caption}
\usepackage{ esint }
\usepackage{stackengine,scalerel,graphicx} 
\stackMath 

\definecolor{trueblue}{rgb}{0.0, 0.45, 0.81}
\definecolor{truegreen}{rgb}{0.13, 0.55, 0.13}

\newcommand{\eps}{\varepsilon} 
\newcommand{\dx}{\, {\rm d}x}

\theoremstyle{plain}
\begingroup
\newtheorem{theorem}{Theorem}[section]
\newtheorem{lemma}[theorem]{Lemma}

\newtheorem{remark}[theorem]{Remark}
\newtheorem{proposition}[theorem]{Proposition}

\endgroup

\newenvironment{step}[1]{\underline{Step #1}.}{}

\theoremstyle{definition}
\begingroup

\endgroup

\renewcommand{\tilde}{\widetilde}

\renewcommand{\d}{ \mathrm{d}}

\DeclareMathOperator{\dist}{dist}

\DeclareMathOperator{\Per}{Per}

\numberwithin{equation}{section}
\newcommand{\N}{\mathbb{N}}

\newcommand{\E}{\mathcal{E}}
\newcommand{\R}{\mathbb{R}}

\renewcommand{\S}{\mathbb{S}}
\renewcommand{\L}{\mathcal{L}}
\renewcommand{\H}{\mathcal{H}}

\newcommand{\mres}{\mathbin{\vrule height 1.6ex depth 0pt width 0.13ex\vrule height 0.13ex depth 0pt width 1.3ex}}

\def\sgn{{\rm sgn}}

\usepackage[normalem]{ulem}

\newcommand{\GGG}{\color{black}}

\def \d{\mathrm{d}}

\begin{document}
	
\title[Energy barriers for boundary nucleation in a two-well model]{Energy barriers for boundary nucleation in a two-well model without gauge invariances}
	
\author[Antonio Tribuzio]{Antonio Tribuzio} 
\address[Antonio Tribuzio]{Institute for Applied Mathematics, University of Bonn, Endenicher Allee 60
53115 Bonn, Germany}
\email{tribuzio@iam.uni-bonn.de}

\author{Konstantinos Zemas}
\address[Konstantinos Zemas]{Institute for Applied Mathematics, University of Bonn, Endenicher Allee 60
	53115 Bonn, Germany}
\email{zemas@iam.uni-bonn.de}
	
\begin{abstract}
We study energy scaling laws for a simplified, singularly perturbed, double-well nucleation problem confined in a half-space, in the absence of gauge invariance and for an inclusion of fixed volume. Motivated by models for boundary nucleation of a single-phase martensite inside a parental phase of austenite, our main focus in this nonlocal isoperimetric problem is how the relationship between the rank-1 direction and the orientation of the half-space influences the energy scaling with respect to the fixed volume of the inclusion. Up to prefactors depending on this relative orientation, the scaling laws coincide with the corresponding ones for bulk nucleation \cite{knupfer2011minimal} for all rank-1 directions, \textit{but} the ones normal to the confining hyperplane, where the scaling is as in a three-well problem in full space, resulting in a lower energy barrier \cite{Tribuzio-Rueland_1}. 
\end{abstract} 
	
\maketitle
	
\section{Introduction}\label{introduction}

In its classical formulation, crystal \emph{nucleation} is the formation of small regions of new thermodynamic phases, for instance in solid materials, that can evolve into macroscopic structures.
In this context, nucleation is usually initiated from a state of \emph{metastability} by a change of temperature or the application of an external load.
From a modeling point of view, critical nuclei can be seen as solutions of minimisation problems involving the competition between a bulk term, favouring the presence of the new phase, and an interfacial energy which penalises interactions between two different phases.
Analysing the energy contributions of the optimal nuclei in terms of the parameters of the system, as \emph{e.g.}\ the volume of the inclusion, helps in quantifying the energy barrier of the phase transformation or the saddle point (between the two phases) of the energy landscape.
Nucleation has been observed as the initiating phenomenon of \emph{martensitic phase transformations} in \emph{shape-memory alloys}, which received a lot of attention in the recent years.
These are materials that, \emph{e.g.},\ during a cooling process, change their crystalline structure from a highly symmetric (\emph{austenite} phase) to a less symmetric one (\emph{martensite} phase).
In the latter, the material may also present different, energetically equivalent crystalline configurations which may lead to the formation of microstructures.
For these transformations, it is not yet completely understood whether (and in which cases) \emph{heterogeneous nucleation} (\emph{i.e.},\ the formation of inclusion emanating from the boundary of the sample) is more favourable than the \emph{homogeneous} one (\emph{i.e.},\ nucleation in the bulk of the sample).

In this paper we contribute in the understanding of the above question, by studying scaling laws for minimal inclusions of a simplified (\emph{i.e.,}\ without rotation or skew invariance), singularly perturbed, double-well energy which models internal energy of martensitic materials.
The inclusions are confined in a half-space, modeling boundary nucleation in shape-memory alloys, and the energy contribution is quantified in terms of the volume of the inclusion $V$ and the strength of the interfacial energy $\eps$ (which can be seen also as a unit length-scale).
We show that, when the two wells are rank-$1$ connected in the direction orthogonal to the boundary hyperplane of the constraint, nucleation is more favourable at the boundary, \emph{i.e.},\ the energy scaling is smaller than that of the unconstrained problem.

To the authors' knowledge, this is the first result in the context of solid-solid phase transitions where energy barriers for boundary nucleation are rigorously obtained, and such a dichotomy in the scaling behavior is observed.

We make also quantitative observations on the separation between the regimes of heterogeneous and homogeneous nucleation, which is governed by the angle between the normal to the boundary of the constraint and the direction of the rank-$1$ connection.

\subsection{Variational theory of martensitic materials}
Before introducing our model, we briefly recall the usual setting of variational models for shape-memory alloys.
A mathematical treatment of martensitic materials, in the context of nonlinear elasticity, has been proposed by Ball and James \cite{BJ87} (see also the notes \cite{M99,B03}).
Given a deformation $u:\Omega\to\R^d$ with reference domain $\Omega\subseteq\R^d$, a prototypical model for the elastic energy associated to $u$ has the form
$$
\int_\Omega \dist^2(\nabla u(x),K)\, \mathrm{d}x\,,
$$
where $K\subset\R^{d\times d}$ is the set of matrices that describe the preferred states of the material and, in general, depends on the temperature.
The quadratic growth of the energy density is suggested by Hooke's law.

In the geometrically nonlinear setting, where rotational invariance under the full group $SO(d)$ is encorporated in the model, the energy-well typically takes the form 
\[K=\bigcup_{i=1}^m SO(d)U_i\,,\]
where $\{U_i\}_{i=1}^m$ are the corresponding stress-free strains for the different phases. The problem of finding all (Lipschitz) solutions $u:\Omega\mapsto \R^d$ to the inclusion
\[\nabla u\in K\] 
is usually referred to as the \textit{$m$-well problem}, and corresponds to finding deformations with zero elastic energy, often called \emph{exact solutions}.
The simplest possible non-trivial solutions are the so called \textit{simple laminates}. The first rigidity result in this direction was by Dolzmann and Müller \cite{DM96}, who, for the nonlinear two-well problem $(m=2)$, showed that if the set where $\nabla u$ is in one of the wells has finite perimeter, then the only exact solutions are in fact simple laminates. Without this assumption, counterexamples via convex integration techniques can be constructed \cite{MS96, MS1999, MS2003}. The result of \cite{DM96} was generalised by Kirchheim \cite{Kirchheim} to the case of three wells.  In \cite{DM96} it is also proven that in the geometrically linearised three-well problem the rigidity holds, without any assumptions on the perimeter of the transition. 

In the study of microstructures in shape memory alloys though, the minimisation of the elastic energy alone cannot give quantitative information, since the continuum model lacks of a natural scale.
For this reason, and also to reduce the degeneracy of the elastic energy, due to its multi-well nature, it is customary to add a higher order term \cite{KM92,KM94,CC15} which penalises transitions between different states and introduce a length-scale to the problem. With the consideration of this extra energy term various quantitative results have been proven, cf. for instance \cite{Lorent, CS2006, JL, CO12, CC15, DV20}; let $\eps>0$ be a small parameter taking the role of a unit length-scale, it is often considered the total energy
$$
\int_\Omega \dist^2(\nabla u(x),K)\, \mathrm{d}x + \eps |D^2 u|(\Omega)
$$
to be a good model for the energy of the system.
In the formula above, the penalisation term is intended as the total variation of the second distributional derivative of $u$.

We are interested in studying minimal inclusions, confined in a half-space and of prescribed volume $V>0$, of a single-state martensite (with preferred strain $F\in\R^{d\times d}$) into the austenite state, which is represented by the zero matrix.
In the nucleation regime, the austenite and martensite states are assumed to have the same energetic cost.
It is customary \cite{CO12, knupfer2011minimal,KKO13} to introduce a geometric variable $\chi\in BV(\R_+^d;\{0,1\})$ whose support represents the region of transformed material, while in the region where $\chi=0$, the material is still in the austenite phase.
We reduce then to the following minimisation problem
\begin{equation}\label{eq:intro1}
E_\eps (V) := \inf_{\substack{\chi\in BV(\R_+^d;\{0,1\}) \\ \|\chi\|_{L^1(\R_+^d)}=V}} \Big\{ \inf_{u\in H^1(\R^d_+;\R^d)} \int_{\R^d_+} |\nabla u - \chi F|^2\, \mathrm{d}x + \eps |D\chi|(\R^d_+) \Big\}\,,
\end{equation}
where we have neglected frame indifference or other gauge-invariances, namely the set of stress-free states is simply $K=\{0,F\}$.
For the sake of keeping the presentation simple, in the introduction we consider only the case in which the constraint is given by $\R^d_+=\{x\in\R^d : x_d>0\}$, but all the other half-spaces are considered in the sequel, obtaining analogous results.

Problems as \eqref{eq:intro1} have been studied mainly in the full-space case.
In particular, we name two works that are strongly connected to the results proved in this article.

\smallskip

\noindent{\em Inclusions for a two-well (geometrically linearized) energy:} in \cite{knupfer2011minimal} the scaling behaviour of nucleation of geometrically linearized (uniform) martensite into austenite is studied in $\R^d$.
The authors considered the quantity
\begin{equation*}\label{eq:intro2-KK}
E^{\rm(sym)}_\eps (V) := \inf_{\substack{\chi\in BV(\R^d;\{0,1\}) \\ \|\chi\|_{L^1(\R^d)}=V}} \Big\{ \inf_{u\in H^1(\R^d;\R^d)} \int_{\R^d} |\nabla^{\rm sym} u - \chi F|^2\, \mathrm{d}x + \eps |D\chi|(\R^d) \Big\}\,,
\end{equation*}
where $\nabla^{\mathrm{sym}}u:=\frac{1}{2}(\nabla u+(\nabla u)^t)$, and proved that when $F$ and $0$ are \emph{strain connected}, \emph{i.e.},\ $F=\frac{1}{2}(n_1\otimes n_2+n_2\otimes n_1)$ for some $n_1,n_2\in\S^{d-1}$, for large volumes $V>\eps^d$, the energy of minimal inclusions scales according to the following law;
\begin{equation}\label{eq:intro3-KKlaw}
E^{\rm(sym)}_\eps(V) \sim \eps^\frac{d}{2d-1}V^\frac{2d-2}{2d-1},
\end{equation}
see \cite[Theorem 2.1]{knupfer2011minimal}.
A construction attaining the optimal scaling law is a thin lens-shaped inclusion which is oriented orthogonally to either $n_1$ or $n_2$.
The gauge-free analogue of this problem has been proven to be equivalent in terms of scaling in \cite{Tribuzio-Rueland_1}.

\smallskip

\noindent{\em Inclusions for three-well energy:} the second result we should mention is also from \cite{Tribuzio-Rueland_1}.
There, a gauge-free three-well problem in $\R^d$ is treated, modelling inclusion of a double-state martensite into austenite; now $\chi$ possibly attains three values $0,1$ and $-1$, meaning that the stress-free states are now $0, F$ and $-F$.
For rank-$1$ matrices $F$, the scaling of
\begin{equation*}\label{eq:intro4-RT}
E^{(3)}_\eps(V) := \inf_{\substack{\chi\in BV(\R^d;\{0,\pm1\}) \\ \|\chi\|_{L^1(\R^d)}=V}} \Big\{ \inf_{u\in H^1(\R^d;\R^d)} \int_{\R^d} |\nabla u - \chi F|^2\, \mathrm{d}x + \eps |D\chi|(\R^d) \Big\}
\end{equation*}
is studied.
Here, internal lamination between $F$ and $-F$ drops the energy barrier with respect to that of uniform martensite inclusions, namely
\begin{equation}\label{eq:intro5-RTlaw}
E^{(3)}_\eps(V) \sim \eps^\frac{2d}{3d-1} V^\frac{3d-3}{3d-1} \ll E^{\rm(sym)}_\eps (V)\,,
\end{equation}
again in the \textit{large volume regime} $V>\eps^d$ (see \cite[Corollary 1.1]{Tribuzio-Rueland_1}). Let us mention that in both cases, in the \textit{small volume regime} $V\leq \eps^d$, interfacial energy dominates and hence both problems behave (in terms of scaling) as the classical isoperimetric problem, resulting in an energy barrier of the form $\eps V^{\frac{d-1}{d}}$.

\smallskip

\noindent{\em Description of the main result:} Our main result (cf.\ Theorem \ref{main_theorem_energy_scaling}) has to be intended as an \textit{interpolation} of the scaling laws \eqref{eq:intro3-KKlaw} and \eqref{eq:intro5-RTlaw} for the problem \eqref{eq:intro1}:
indeed, if the direction of the rank-$1$ connection between $F$ and $0$, denoted by $n\in\S^{d-1}$, is not parallel to $e_d$, then for large enough volumes, \[E_\eps(V)\sim\eps^\frac{d}{2d-1} V^\frac{2d-2}{2d-1}\,,\]
if instead $n=\pm e_d$, then
$$
E_\eps(V)\sim\eps^\frac{2d}{3d-1} V^\frac{3d-3}{3d-1}.
$$
In the former case, the energy barrier is up to prefactors depending on the relative orientation of the rank-$1$ direction and the confining hyperplane.
In the latter case, the lack of the \emph{twin} state $-F$ is covered by the presence of the constraint, suitably oriented.
This dichotomy is in fact made quantitative in terms of $\eps$ and $V$ in the upper bound construction of Subsection \ref{sec:4.2}, in the sense that if the minimal angle between $n$ and $e_d$ is below a critical threshold, then the energy barrier scales as in \eqref{eq:intro3-KKlaw}, otherwise as in \eqref{eq:intro5-RTlaw}.
For further discussions about this dichotomy, see Section \ref{setting_main_result} below.

\subsection{Relations to the literature and other quantitative results}

From a mathematical point of view, \eqref{eq:intro1} is a variant of the classical isoperimetric problem (in half-space) with the presence of a nonlocal and strongly anisotropic bulk term.
Indeed, \eqref{eq:intro1} reads as
$$
E_\eps(V) := \inf\Big\{E_{el}(E)+\Per(E;\R^d_+) : E\subseteq\R^d_+, \mathcal{L}^d(E)=V\Big\}\,,
$$
where the elastic energy term
$$
E_{el}(E) := \inf_{u\in H^1(\R^d_+;\R^d)} \int_{\R^d_+}|\nabla u- \chi_E F|^2\, \mathrm{d}x
$$
penalises misorientations of the outer normal vector to the boundary of $E$ from $n$.
Existence and fine properties of minimisers for closely related functionals have been studied in \cite{KN18,KS22}, for nucleation of ferromagnetic domains. 
Similar variants of these isoperimetric problems related to liquid drop models, with a nonlocal (Coulomb type) bulk term have been studied for instance in \cite{CS12,KnuMur13,KnuMur14,NP21}.

In the context of shape-memory alloys, apart from the works \cite{knupfer2011minimal,Tribuzio-Rueland_1} that are our main references, nucleation processes have been studied for the geometrically linearised cubic-to-tetragonal phase transformation both in the local \cite{CO12} and in the global case \cite{KKO13,KO18}.
Recently a two-dimensional, geometrically nonlinear analogue of \cite{knupfer2011minimal} has been analysed in \cite{AKKR23}.
Scaling laws for boundary (corners) nucleation in the cubic-to-tetragonal phase transformation has been studied in \cite{BG13}.
The emergence of boundary nucleation is also interrogated in \cite{BallKoumatosSeiner13,BallKoumatos16,James19}.
Other interesting related results are obtained, in a double-well model with an \emph{a priori} prescribed shape of the domain \cite{CDMZ20} or in a fixed domain with one of the variants having low volume fraction \cite{CZ16,CDZ17}.

The quantitative study of microstructures emerging in martensitic materials, through \emph{e.g.}\ scaling laws of multi-well singularly perturbed elastic energies, is nowadays a broad subject; starting from the pioneering results from Kohn and M\"uller \cite{KM92,KM94} it received an always increasing attention.
We give here a (far from being complete) list of references for the interested readers.
Different types of scaling laws have been obtained in fixed domains with prescribed boundary conditions in nonlinear \cite{CC15} and geometrically linearised settings \cite{CO09,Rueland16}.
In these works, the complexity of the emerging microstructures, and thus of the scaling, may be affected by different factors as the order of lamination of the boundary conditions \cite{RT22Tartar,RT23ESAIM}, the geometry of the domain \cite{GZ23,G23+} or by the multiplicity of the connections between the wells \cite{CC15,RRT23,RRTT23+}.
Scaling laws can be used to detect finer properties of minimisers as asymptotic self-similarity \cite{Conti2000}, compactness \cite{Simon21} or maximal regularity \cite{RTZ19,RZZ20}.

In the end, we comment that related results have also been studied in other physical systems such as in compliance minimisation \cite{KW14,KW16} or micromagnetism \cite{CKO99}. \\[-13pt]

\smallskip

\textit{Plan of the paper}. After fixing some notation in the next subsection, in Section \ref{setting_main_result} we introduce in detail our model of study and state our main result, Theorem \ref{main_theorem_energy_scaling}. In Section \ref{sec: 3} we prove an auxiliary local lower bound for the elastic energy of a pair $(u,\chi)$ for an inclusion with small volume and interfacial energy, in the spirit of \cite[Proposition 3.1]{knupfer2011minimal}. In Section \ref{sec:4} we prove our main theorem in the \textit{non-degenerate case} corresponding to $n\neq \{\pm e_d\}$, by providing an \textit{ansatz-free lower bound} (based on the result of Section \ref{sec: 3} and a standard covering argument), and also a matching in terms of energy scaling, upper bound construction. Section \ref{sec:reflection} contains the proof in the \textit{degenerate case} $n=\{\pm e_d\}$, by means of a reflection argument and a direct relation of the resulting problem with a three-well problem in $\R^d$, for which the results of \cite{Tribuzio-Rueland_1} can be employed. Some elementary geometric calculations needed in Section \ref{sec: 3} are collected in Appendix \ref{sec: choice_of_cages}.
	\\[-10pt]
\subsection{Notation and Preliminaries}\label{Notation}
For $d\in \N, d\geq 2$, by $\mathcal{B}(\R^d)$ and $\mathcal{M}(\R^d)$ we denote the $\sigma$-algebra of Borel and Lebesgue measurable subsets of $\R^d$ respectively. By the symbols $\cdot$ and $|\, |$ we denote the Euclidean inner product and norm in $\R^d$ with respect to the canonical basis $\{e_1,\dots,e_d\}$ respectively. The corresponding notions in the space of $\R^{d\times d}$-matrices are denoted by $\colon$ and $|\,|$. For every $\xi\in \S^{d-1}$ we denote 
\begin{equation}\label{eq: pi_nu}
\Pi_\xi:=\{x\in \R^d\colon x\cdot \xi= 0\} \ \ \text{ and } \ \ \Pi_{\xi,\pm}:=\{x\in\R^d\colon \pm x\cdot \xi> 0\}\,,
\end{equation}
to be respectively the hyperplane and upper$\slash$lower-half space with respect to the outward pointing unit normal $\xi$. With this notation we have $\R^d_{\pm}=\Pi_{e_d,\pm}$.
Given $\rho>0$ and $x_0\in \R^d$, we denote $B_\rho^{\pm}(x_0):=B_\rho(x_0)\cap \R^d_{\pm}\,.$ 
For a function $f\in L^1(\R_+^d)$ we denote by $\mathrm{spt}(f)$ the support of $f$, \textit{i.e.},
\[\mathrm{spt}(f):=\overline{\{x\in \R_+^d\colon f(x)\neq 0\}}\,.\]
Given $A, B\subset \R^d$ we write $\chi_A$ for the indicator function of $A$, \textit{i.e.}, 
\[\chi_A(x):=\begin{cases} 1\,, \quad \text{if } x\in A\,\\
0\,,  \quad \text{if } x\notin A\,,
\end{cases}\]
and $A \subset \subset B$ iff $\overline{A} \subset B$. For $d,k\in \mathbb{N}$ we denote by $\mathcal{L}^d$ and $\mathcal{H}^{k}$ the $                                                                                                                                                                                                                                                                                                                                                                                                                                                                                                                                                                                                                                                                                                                                                                                                                                                                                                                                                                                                                                                                                                                                                                                                                                                                                                                                                                                                                                                                                                                                                                                                                                                                                                                                                                                                                                                                                                                                                                                                                                                                                                                                                                                                                                                                                                                                                                                                                                                                                                                                                                                                                                                                                                                                                                                                                                                                                                                                                                                                                                                                                                                                                                                                                                                                                                                                                                                                                                                                                                                                                                                                                                                                                                                                                                                                                                                                                                                                                                                                                                                                                                                                                                                                                                                                                                                                                                                                                                                                                                                                                                                                                                                                                                                                                                                                                                                                                                                                                                                                                                                                                                                                                                                                                                                                                                                                                                                                                                                                                                                                                                                                                                                                                                                                                                                                                                                                                                                                                                                                                                                                                                                                                                                                                                                                                                                                                                                                                                                                                                                                                                                                                                                                                                                                                                                                                                                                                                                                                                                                                                                                                                                                                                                                                                                                                                                                                                                                                                                                                                                                                                                                                                                                                                                                                                                                                                                                                                                                                                                                                                                                                                                                                                                                                                                                                                                                                                                                                                                                                                                                                                                                                                                                                                                                                                                                                                                                                                                                                                                                                                                                                                                                                                                                                                                                                                                                                                                                                                                                                                                                                                                                                                                                                                                                                                                                                                                                                                                                                                                                                                                                                                                                                                                                                                                                                                                                                                                                                                                                                                                                                                                                                                                                                                                                                                                                                                                                                                                                                                                                                                                                                                                                                                                                                                                                                                                                                                                                                                                                                                                                                                                                                                                                                                                                                                                                                                                                                                                                                                                                                                                                                                                                                                                                                                                                                                                                                                                                                                                                                                                                                                                                                                                                                                                                                                                                                                                                                                                                                                                                                                                                                                                                                                                                                                                                                                                                                                                                                                                                                                                                                                                                                                                                                                                                                                                                                                                                                                                     d$-dimensional Lebesgue measure and the $k$-dimensional Hausdorff measure, respectively. Moreover, we use the notation $\omega_d:=\L^d(B_1)$ for the volume of the unit ball in $\R^d$. 
For every $x,y\in \R^d$ we denote by $[x,y]$ the directed segment in $\R^d$ with endpoints $x$ and $y$, and for every $A\subset \R^d$ we denote by $\mathrm{conv}(A)$ its convex hull, \textit{i.e.}, the smallest (with respect to set inclusion) convex set in $\R^d$ that contains $A$.

By $\sim_{M_1,M_2,\dots}, \lesssim_{M_1,M_2,\dots}$ we mean that the corresponding equality or inequality is valid up to a constant that depends only on the parameters $M_1, M_2,\dots$ and the dimension, but is allowed to vary from line to line. By universal constant we mean a constant that depends at most on the dimension. We will also use the notation $f\sim g$ whenever there exist two universal constants $0<c\leq C<+\infty$, so that $$c g\leq f \leq C g\,.$$ The symbols $\ll$ and $\gg$ mean that the corresponding estimate requires a small universal constant, for example, if we say that $f\lesssim g$ for $0<\eps\ll 1$, this will mean that there exist a universal constant $C>0$ and a sufficiently small universal constant $\eps_0\in (0,1)$, so that 
\[f\leq Cg\, \quad \forall \eps\in (0,\eps_0]\,.\]
Given a Lebesgue measurable function $f$ and a bounded set $E\subset \mathcal{M}(\R^d)$, we denote 
\[
\langle f\rangle_E:=\fint_{E}f:=\begin{cases}\frac{1}{\mathcal{L}^d(E)}\int_{E}f\,\mathrm{d}x\, \quad \mathrm{if }\ \L^d(E)>0\,,\\[4pt]
0\, \ \qquad \qquad \qquad \mathrm{if }\ \L^d(E)=0\,.	
\end{cases}
\]
If $U\subset \R^d$ is open and $u\in BV(U)$ (cf. \cite[Chapter 3]{Ambrosio-Fusco-Pallara:2000} for the definition and a systematic treatment of the subject), we denote by $|Du|(U):=\int_U\,\d|Du|$ its total variation in $U$.  
The set of rank-$1$ matrices is denoted by
\begin{equation}\label{eq:_compatible_strains}
\mathcal{R}_1(d):=\Big\{F\in \R^{d\times d}\colon F=b\otimes n\,\ \text{for some } b\in \R^d\setminus \{0\}\,,\ n\in \S^{d-1}\Big\}\,,
\end{equation} 
where the tensor product is defined componentwise, \textit{i.e.}, $(b\otimes n)_{ij}:=b_in_j \ \forall i,j=1,\dots,d$. 
We also set as usual
\[SO(d):=\big\{R\in \R^{d\times d}\colon \ R^tR=I_d\, \ \text{and } \mathrm{det}R=1\big\}\,,\]	
to be the special orthogonal group in $d$ dimensions.
	
\section{Setting and statement of the main result}\label{setting_main_result} 
Given $\eps>0$, $\xi\in \S^{d-1}$ and $F=b\otimes n\in \mathcal{R}_1(d)$, we consider the energy 
\begin{equation}\label{def: eps_energy_on_pairs}
\mathcal{E}_{\eps}(u,\chi ):=\int_{\Pi_{\xi,+}}\big|\nabla u-\chi F\big|^2\,\mathrm{d}x+\eps |D\chi|(\Pi_{\xi,+})\,,
\end{equation}
where we recall the notation in  \eqref{eq: pi_nu}, \eqref{eq:_compatible_strains}. 
The two preferred phases are represented by $0_{d\times d}$ (the \textit{majority$\backslash$austenite phase}) and $F$ (the \textit{minority$\backslash$martensite phase}). The function $u\in H^1(\Pi_{\xi,+}; \R^d)$ represents the elastic displacement, so that the first term in \eqref{def: eps_energy_on_pairs} represents the corresponding linearised elastic energy, neglecting any gauge invariance in this toy model. 
	
The two-valued function $\chi\in BV(\Pi_{\xi,+};\{0,1\})$ describes the region occupied by the minority phase. By the general theory for sets of finite perimeter (cf. \cite[Chapter 3]{Ambrosio-Fusco-Pallara:2000}), the set 
\begin{equation}\label{eq: jump_set_of_chi}
J_\chi:=(\partial^*\{\chi=1\})\cap \Pi_{\xi,+}
\end{equation} is $(d-1)$-rectifiable, and if $\nu_\chi$ denotes its measure-theoretical outward pointing unit normal, then
$$D\chi=(\chi^+-\chi^-) \nu_\chi\mathcal{H}^{d-1}\mres J_\chi\,,$$
where $\chi^{\pm}$ denote as usual the measure theoretic traces (taking values 0 or 1 in this case).  
	
In particular, prescribing the volume $V>0$ of the inclusion, the set of admissible configurations is defined as 
\begin{equation}\label{admissible_configurations}
\mathcal{A}_{\xi}(V):=\Big\{(u,\chi)\in H^1(\Pi_{\xi,+};\R^d)\times BV(\Pi_{\xi,+};\{0,1\})\colon \int_{\Pi_{\xi,+}}\chi\,\mathrm{d}x=V\Big\}\,,
\end{equation}
and the minimal energy for an inclusion at fixed volume $V>0$ is defined as
\begin{equation}\label{minimal_energy_fixed_volume}
E_\eps(V):=\inf_{\chi}\big[\inf_{u} \mathcal{E}_{\eps}(u,\chi)\big] \ \ \text{among all } (u,\chi)\in \mathcal{A}_{\xi}(V)\,.
\end{equation}
With these definitions, the main result of this paper is the following.
	
\begin{theorem}[\bf Scaling of the nucleation barrier]\label{main_theorem_energy_scaling}
Let $d\in\N$, $d\geq 2$, let $F=b\otimes n$ for some $b\in\R^d\setminus\{0\}$ and $n\in\S^{d-1}$, $\eps,V>0$ and let $E_\eps(V)$ be as in \eqref{minimal_energy_fixed_volume}. 
Then we have the following dichotomy: 
\begin{itemize}
\item[$\rm{(i)}$] For every $n\in \S^{d-1}\setminus\{\pm \xi\}$, we have 
\begin{equation}\label{eq: generic_energy_scaling}
E_\eps(V)\sim_{n,\xi} \begin{cases} \eps V^{\frac{d-1}{d}}\quad \quad \quad \quad \quad \quad \quad   \text{if } V\leq \eps^d|F|^{-2d}\,,\\[4pt]
\eps^{\frac{d}{2d-1}}|F|^{\frac{2d-2}{2d-1}}V^{\frac{2d-2}{2d-1}}\ \ \ \ \ \text{if } V\geq \eps^d|F|^{-2d}\,.
\end{cases}
\end{equation}
\item[$\rm{(ii)}$] For $n=\pm \xi$, we have 
\begin{equation}\label{eq: specific_direction_energy_scaling}
E_\eps(V)\sim \begin{cases}
 \eps V^{\frac{d-1}{d}}\quad \quad \quad \quad \quad \quad \quad   \text{if } V\leq \eps^d|F|^{-2d}\,,\\[4pt]
\eps^{\frac{2d}{3d-1}}|F|^{\frac{2d-2}{3d-1}}V^{\frac{3d-3}{3d-1}}\ \ \ \ \ \text{if } V\geq \eps^d|F|^{-2d}\,.
\end{cases}
\end{equation}
\end{itemize}
\end{theorem}
	
As we also discussed in the Introduction \ref{introduction}, equations \eqref{eq: generic_energy_scaling} and \eqref{eq: specific_direction_energy_scaling} give the scaling of the minimal energy at fixed volume. Note that for small volumes the interfacial energy dominates, hence the scaling of $E_\eps(V)$ with respect to $V$ corresponds to the one of the relative isoperimetric problem in half-space. By a simple reflection argument the latter coincides with the scaling in $\R^d$, independently of the orientation of the half-space. In particular, the optimal scaling is achieved by an inclusion having the shape of a (half-) ball. 

For large volumes the two terms in \eqref{def: eps_energy_on_pairs} compete with each other, resulting in a change of scaling (with respect to $V$) for the optimal energy. 
The interesting feature of this \textit{constrained in $\Pi_{\xi,+}$} nonlocal, strongly anisotropic isoperimetric problem then, is that, unlike the pure relative isoperimetric problem in half-space, the relationship between the direction of the half-space (indicated by the normal $\xi\in \S^{d-1}$) and the rank-1 direction $n\in \S^{d-1}$ influences the energy scaling. 

Interestingly, the latter agrees with the corresponding scaling for the problem in full space (cf. \cite[Theorem 2.1]{knupfer2011minimal}  and \cite[Theorem 1]{Tribuzio-Rueland_1}) 
for all possible rank-1 directions but the particular ones, namely when $n=\pm\xi$. In this special case, boundary nucleation for the minority phase is energetically favorable than nucleation in the bulk, since by means of a reflection argument, the problem in this case transforms to a problem in $\R^d$, but for a three-well compatible problem without gauge invariance (cf. \cite{Tribuzio-Rueland_1}). 

Indeed, choosing for simplicity $b\in \S^{d-1}$, so that $|F|=1$, we have that 
\[V\geq \eps^d\iff \eps^{\frac{2d}{3d-1}}V^{\frac{3d-3}{3d-1}}\leq  \eps^{\frac{d}{2d-1}}V^{\frac{2d-2}{2d-1}}\,,\]
\textit{i.e.}, \eqref{eq: specific_direction_energy_scaling} for large volumes gives a lower energy that the one predicted for nucleation in the bulk as in \eqref{eq: generic_energy_scaling}.

In what follows in the analysis, we will without restriction assume that 
\begin{equation}\label{eq: reductions_for_plane_and_eps}
\xi=e_d \ \text{ and } \ \eps=1\,.	
\end{equation}
Indeed, if $Q_\xi\in SO(d)$ such that $\xi=Q_\xi e_d$, setting
\begin{align*}
\begin{split}
u_{\xi,\eps}(\tilde x)&:=u(\eps^{-1}Q_\xi\tilde x)\,,\qquad \qquad \qquad  \ \ \ \ \qquad \chi_{\xi,\eps}(\tilde x):=\chi(\eps^{-1}Q_\xi\tilde x)\,, \\
\ \ F_{\xi,\eps}&:=\eps^{-1}FQ_\xi= (\eps^{-1} b)\otimes (Q_\xi^t n)\,, \qquad \ \  V_\eps:=\eps^d V\,,
\end{split}
\end{align*}
we have that $(u_{\xi,\eps}, \chi_{\xi,\eps})\in \mathcal{A}_{e_d}(V_\eps)$. By the change of variables $x:=\eps^{-1}Q_\xi\tilde x$, so that 
\[\tilde x\in\R^d_+\iff x\in \Pi_{\xi,+}\,,\]
we also have 
\begin{equation*}
\mathcal{E}_{\eps}(u,\chi)=\eps^{2-d}\mathcal{E}_1(u_{\xi,\eps},\chi_{\xi,\eps})\implies E_\eps(V)=\eps^{2-d}E_1(V_\eps)\,.
\end{equation*}
In that respect, for notational convenience we will henceforth denote 
\begin{align}\label{relabeled_variables}
\begin{split}	
&\nu:=Q_\xi^t n\,, \ \ G:=a\otimes \nu:= F_{\xi,\eps}\,, \ \ \mu:=V_\eps>0\,, \\[3pt]
&\mathcal{A}(\mu):=\mathcal{A}_{e_d}(\mu)\,, \ E(\mu):=E_1(\mu)\,,
\end{split}
\end{align} 
where we omit the dependence on $\eps$ from the notation of $G,\mu$. Then, Theorem \ref{main_theorem_energy_scaling} becomes equivalent to its following version.
	
\begin{theorem}[\bf Rotated and rescaled version of Theorem \ref{main_theorem_energy_scaling}]\label{main_theorem_rescaled}
In any dimension $d\geq 2$, and with the notation introduced in \eqref{relabeled_variables}, we have the following dichotomy: 

\begin{itemize}
\item[$\rm{(i)}$] For every $\nu\in \S^{d-1}\setminus\{\pm e_d\}$, we have 
\begin{equation}\label{eq: rescaled_generic_energy_scaling}
E(\mu)\sim_{\nu} \begin{cases} \mu^{\frac{d-1}{d}} \ \qquad \qquad \quad \quad   \text{if } \mu\leq |G|^{-2d}\,,\\[4pt]
|G|^{\frac{2d-2}{2d-1}}\mu^{\frac{2d-2}{2d-1}} \qquad \ \ \text{if } \mu\geq|G|^{-2d}\,.
\end{cases}
\end{equation}
\item[$\rm{(ii)}$] For $\nu=\pm e_d$, we have 
\begin{equation}\label{eq: rescaled_specific_direction_energy_scaling}
E(\mu)\sim \begin{cases}
\mu^{\frac{d-1}{d}}\qquad  \qquad \qquad   \text{if } \mu\leq |G|^{-2d}\,,\\[4pt]
|G|^{\frac{2d-2}{3d-1}}\mu^{\frac{3d-3}{3d-1}}\quad \ \ \ \ \text{if } \mu\geq |G|^{-2d}\,.
	\end{cases}
\end{equation}
\end{itemize}
\end{theorem}
Therefore, in the next sections we focus on proving Theorem \ref{main_theorem_rescaled}.
	
\section{Lower bounds for the elastic energy}\label{sec: 3}
As already explained at the end of the previous section, without restriction we now consider the energy
	
\begin{equation}\label{eq: rescaled_energy}
\mathcal{E}(u,\chi):=\int_{\R^d_+}|\nabla u-\chi G|^2\,\mathrm{d}x+|D\chi|(\R^d_+)\,,
\end{equation}
where $G=a\otimes \nu\in \mathcal{R}_1(d)$, among admissible pairs $(u,\chi)\in \mathcal{A}(\mu)$, where
\begin{equation}\label{eq: rescaled_admissible_pairs}
\mathcal{A}(\mu):=\bigg\{(u,\chi)\in H^1(\R^d_+;\R^d)\times BV(\R^d_+;\{0,1\})\colon\ \int_{\R^d_+}\chi\,\mathrm{d}x=\mu\bigg\}\,.
\end{equation}
For every $K\in \mathcal{M}(\R^d_+)$ we use the shorthand notation
\[\E((u,\chi);K):=\int_{K}|\nabla u-\chi G|^2\,\mathrm{d}x+|D\chi |(K)\,,\]
omitting the dependence when $K\equiv\R^d_+$, and we also set
\begin{equation}\label{eq: rescaled_minimal_energy}
E(\mu):=\inf_{\chi}\big[\inf_{u} \mathcal{E}(u,\chi)\big] \ \ \text{among all } (u,\chi)\in \mathcal{A}(\mu)\,.
\end{equation}
	
In this section we prove a lower bound for the bulk elastic energy of a pair $(u,\chi)\in \mathcal{A}(\mu)$ for an inclusion with small volume and interfacial energy on \textit{spherical caps-type} sets based on the boundary hyperplane $\R^d_+$. The statement and main idea of proof are analogous to \cite[Proposition 3.1]{knupfer2011minimal} in the context of linearised elasticity in $\R^d$, but the restriction on $\R^d_+$ here leads to some adjustments in the arguments.

In particular, the rank-$1$ direction $\nu$ will play a crucial role for an appropriate modification of the original covering argument, since the cages used therein will be appropriately tilted.
From a technical point of view, in contrast to the proof of \cite[Proposition 3.1]{knupfer2011minimal}, here we will use the smallness of the elastic energy, $\chi$ and $D\chi$ \textit{only} on the upper horizontal face of the cages, and not on the tilted boundaries and the lower horizontal face, which we are not allowed to \textit{wiggle} anyway because of the constraint.

Again, the proof will be given first in 2 dimensions (see Step 2 below), and then by means of an inductive argument also in higher dimensions $d\geq 3$ (see Step 3). 
	
\begin{proposition}\label{elastic_lower_bound_on_half_balls}
\normalfont Let $d\geq 2$ and $G=a\otimes \nu\in \mathcal{R}_1(d)$ with $\nu\neq \pm e_d$. There exist constants $ c, \gamma\in (0,1)$ only depending on $d$ and $\nu$ with the following property. For every $z\in \R_+^d$, $\rho>0$  
and $\chi\in BV(B^+_\rho(z); \{0,1\})$ with 
\begin{equation}\label{small_volume_and_perimeter_of_minority}
\|\chi\|_{L^1(B^+_\rho(z))}\leq c\rho^d\quad \text{ and }\quad |D\chi|(B^+_\rho(z))\leq c \rho^{d-1}\,, 
\end{equation}
there holds
\begin{equation}\label{1st_lower_bound_for_elastic_energy}
\inf_{u\in H^1(\R^d_+;\R^d)}\int_{B_\rho^+(z)}|\nabla u-\chi G|^2\,\mathrm{d}x\geq  c \rho^{-d}|G|^2\|\chi\|^2_{L^1(B^+_{\gamma \rho}(z))}\,.
\end{equation}
\end{proposition}

\begin{proof}
\begin{step}{\bf (0)} \underline{(Reducing to $z=z_de_d, \rho=1, \nu\in \mathrm{span}\{e_1,e_d\}$)} For the proof of the proposition we can assume without loss of generality that  
\begin{equation}\label{eq: 1st reduction by shift and}
z=z_de_d, \ \ \rho=1 \text{ and that } \ \nu=\nu_1e_1+\nu_de_d \ \text{ with }\nu_1\neq 0\,.
\end{equation} Indeed, for every $(u, \chi)\in H^1(\R^d_+;\R^d)\times BV(B_\rho^+(z);\{0,1\})$ we can consider the pair $(u_{z,\rho,\nu},\chi_{z,\rho,\nu})\in H^1(\R^d_+;\R^d)\times BV(B_1^+\big((\rho^{-1}z_d)e_d);\{0,1\}\big)$, defined via
\[u_{z,\rho,\nu}(y):=u(\hat z+\rho R_\nu y)\,, \ \ \chi_{z,\rho,\nu}(y):=\chi(\hat z+\rho R_\nu y)\,,\]
where $\hat z:=z-z_de_d$,  $R_\nu\in SO(d)$ is such that $R_\nu e_1=\frac{\hat \nu}{|\hat \nu|}$ and $R_\nu e_d=e_d$, where again $\hat\nu:=\nu-\nu_de_d\neq 0$ is the (non-normalised) projection of $\nu$ onto $\R^{d-1}\times \{0\}$. Considering also 
\[G_{\rho,\nu}:=\rho GR_\nu=(\rho a)\otimes (|\hat \nu|e_1+\nu_d e_d)\in \mathcal{R}_1(d)\,,\] 
we see that if \eqref{1st_lower_bound_for_elastic_energy} holds for $(u_{z,\rho,\nu},\chi_{z,\rho,\nu}, G_{\rho,\nu})$ with respect to $B_1^+\big((\rho^{-1}z_d)e_d\big)$ and $B_{\gamma}^+\big((\rho^{-1}z_d)e_d\big)$, then by the change of variables $x:=\hat z+\rho R_\nu y$, it holds true as well for $(u,\chi,G)$ with respect to $B_\rho^+(z)$ and $B_{\gamma\rho}^+(z)$\,. 
\end{step}\\[1pt]

\begin{step}{\bf (1)} \underline{(Reducing further to the case $0\leq z_d\leq \theta$ for some $0<\theta\ll 1$)} We can further assume without restriction that for some $\theta>0$ sufficiently small to be chosen in the next steps, the center of the spherical-type cap satisfies 
\begin{equation}\label{theta_height}
0\leq z_d\leq \theta\,.	
\end{equation}
Indeed, in the case $z_d>\theta$, we have that $B_\theta(z_de_d)\subset \R^d_+$, and the localisation argument of \cite[Claim 3.1]{Tribuzio-Rueland_1} (similar to that of \cite[Proposition 3.1]{knupfer2011minimal}) can be applied in the unconstrained case in the ball $B_{\theta}(z_de_d)\subset B_1^+(z_de_d)$. Indeed, by taking $c>0$ small enough so that $c\leq c_d\theta^d\leq c_d\theta^{d-1}$, we see that \eqref{small_volume_and_perimeter_of_minority} (for $z:=z_de_d$, $\rho:=1$) implies that conditions \cite[(16) of Claim 3.1]{Tribuzio-Rueland_1} are satisfied. 
Therefore, for the dimensional constant $\alpha_d>0$ therein, we obtain
\begin{align}\label{lower_bound_from_knupfer_kohn}
\begin{split}	
\inf_{u\in H^1(\R^d_+;\R^d)}\int_{B_1^+(z_de_d)}|\nabla u-\chi G|^2&
\geq \inf_{u\in H^1(\R^d_+;\R^d)}\int_{B_\theta(z_de_d)}|\nabla u-\chi G|^2\\[3pt]
&\geq  c_d \theta^{-d}|G|^2\|\chi\|^2_{L^1(B_{\alpha_d\theta}(z_de_d))}\,.
\end{split}
\end{align} 
\end{step}

\begin{step}{\bf (2)} \underline{(Proof in the case $d=2$)} Without loss of generality we further assume that $\nu:=(\nu_1,\nu_2)\in \mathbb{S}^1$ is such that $\nu_2\leq 0<\nu_1$, since all other cases are completely analogous. Abbreviating notation, we also set $\vec{z}_2:=z_2e_2$.
Let $\theta,c,\gamma\in (0,1)$ be suitable fixed constants to be specified later, depending only on $\nu$ at this step.
We assume by contradiction that there exists a $\chi$ complying with \eqref{small_volume_and_perimeter_of_minority} for which the inequality in \eqref{1st_lower_bound_for_elastic_energy} does not hold (both for $\rho=1$), \textit{i.e.}, there exists $u\in H^1(\R_+^2;\R^2)$ such that
\begin{equation}\label{eq: basic_contradiction_inequality_1}
\int_{B_1^+(\vec{z}_2)}|\nabla u-\chi G|^2\, \mathrm{d}x<c|G|^2\mu_\gamma^2\,, \ \text{where } \mu_\gamma:=\int_{B_\gamma^+(\vec{z}_2)}\chi\,\mathrm{d}x\,.
\end{equation}
Setting $\nu^\bot:=(-\nu_2,\nu_1)$  and since $G\nu^\bot:=(\nu\cdot\nu^\bot)a=0$, \eqref{eq: basic_contradiction_inequality_1} also implies 
\begin{equation}\label{eq: basic_contradiction_inequality_2}
\int_{B_1^+(\vec{z}_2)}\big|\partial_{\nu^\bot}u\big|^2\, \mathrm{d}x<c|G|^2\mu_\gamma^2\,.
\end{equation}
\end{step}
\begin{figure}[t]
\begin{center}
\includegraphics{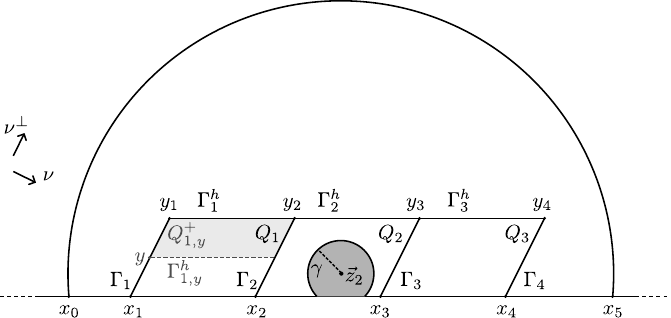}
\caption{An illustration of the \emph{tilted cages} defined in step {\bf(2a)}.}
\label{fig:cage1}
\end{center}
\end{figure}
\begin{step}{\bf (2a)} \underline{(Notation and choice of tilted cages)}
Let us set $x_0:=\big(-\sqrt{1-z_2^2},0\big)$ and $x_5:=\big(\sqrt{1-z_2^2},0\big)$ so that 
\begin{equation}\label{base_cap_endpoints}
|x_5-x_0|=2\sqrt{1-z_2^2}\in [2\sqrt{1-\theta^2},2]\,.
\end{equation}
Analogously to Step 1 in the proof of \cite[Lemma 3.2]{knupfer2011minimal}, one can find $\sigma\in[2\gamma,4\gamma]$ and an ordered quadruple of points $\{x_1,x_2,x_3,x_4\}$ in the $x$-axis, such that 
\begin{align}\label{properties_of_tilting}
\begin{split}	
\mathrm{(i)}&\quad \mathrm{For\ all\ } i=1,2,3\colon \quad |x_{i+1}-x_i|\sim 1 
\,,\\[1.5pt]
\mathrm{(ii)}&\quad B_{\gamma}^+(\vec{z}_2)\subset Q_2\subset Q\subset B_1^+(\vec{z}_2)\,,\\[1.5pt]
\mathrm{(iii)}& \quad \chi=0 \quad  \H^1\text{-a.e.} \text{ on } \Gamma^h
\,,\\[1.5pt]
\mathrm{(iv)}& \quad \int_{\Gamma^h}|\nabla u|^2\,\mathrm{d}\H^1
<\frac{c}{\gamma}|G|^2\mu_\gamma^2\,.
\end{split}
\end{align} 
Here we have used the notation:
\begin{align}\label{eq: stupid notation}
\begin{split}
\forall \sigma>0\colon &\quad \R_{\sigma}^h=\R\times\{\sigma\} \quad \mathrm{and } \quad  S_{\sigma}:=\{x\in \R^2\colon 0\leq x_2\leq \sigma\}\,,\\
\forall i=1,2,3,4\colon &\quad \ \Gamma_i:=(x_i+\R\nu^\bot)\cap S_\sigma \quad  \mathrm{and } \ \ y_i:=(x_i+\R\nu^\bot)\cap \R_{\sigma}^h\,,\\
\forall i=1,2,3\colon &\quad \ \Gamma^h_i:=[y_i,y_{i+1}]\,, \quad Q_i:=\mathrm{conv}\big(\{x_i,x_{i+1},y_{i+1}, y_i\}\big)\,,\\
&\quad \ \Gamma^h:=\bigcup_{i=1}^3 \Gamma^h_i\,, \quad Q:=\bigcup_{i=1}^3 Q_i \,, 
\end{split}
\end{align}
see Figure \ref{fig:cage1}. The proof of \eqref{properties_of_tilting} is elementary and will be given in Appendix \ref{sec: choice_of_cages}.\\[1pt]
In what follows we will make the convention that the constant $c>0$ is allowed to vary from line to line without relabelling, as long as the possibly new constant is the same as the previous one, up to multiplicative prefactors that are independent of $\nu$.
\end{step}	
\\[-3pt]
 
\begin{step}{\bf (2b)} \underline{($u_1$ is small on average in all cages.)} By the invariance of the energy under translations we may also assume without restriction that $\langle u_1\rangle_{\Gamma^h}=0$.
Then, by \eqref{properties_of_tilting}(iii),(iv), we get
\begin{align}\label{eq:Linfty_estimate}
\|u_1\|_{L^\infty(\Gamma^h)}\leq \int_{\Gamma^h}|\partial_1u_1|\ \mathrm{d}\H^1\leq [\L^1(\Gamma^h)]^{\frac{1}{2}}\|\nabla u\|_{L^2(\Gamma^h)} \lesssim \sqrt{\frac{c}{\gamma}}|G|\mu_\gamma\,,
\end{align}
since $\L^1(\Gamma^h)\sim 1$. We now claim that 
\begin{equation}\label{small_averages}
\big|\langle u_1\rangle_{Q_i}\big|\lesssim 
\sqrt{c}\bigg(\frac{1}{\sqrt\gamma}+\frac{\sqrt{\gamma}}{\nu_1}\bigg)|G|\mu_\gamma\,\ \ \forall i=1,2,3\,.
\end{equation}
For this purpose, $\forall y\in \Gamma_i$ let  $\Gamma^h_{i,y}:=(y+\mathbb{R}e_1)\cap Q_i$, and $Q_{i,y}^+:=\mathrm{conv}(\Gamma_{i,y}^h\cup\{y_i,y_{i+1}\})$ (cf. Figure \ref{fig:cage1}) be the corresponding upper sub-cage, with generalised outward pointing unit normal $\nu_{\partial Q_{i,y}^+}$. Note that 
\begin{align*}
\nu_{\partial Q_{i,y}^+}=\begin{cases}
-\nu \ \ \ \ \text{on } \ \Gamma_i\,\\[2pt]
\nu \quad \ \  \ \text{on } \  \Gamma_{i+1}\,\\[2pt]
-e_2 \ \ \ \text{on } \ \Gamma_{i,y}^h\,\\[2pt]
e_2 \quad \ \ \text{on } \ \Gamma_i^h\,,
\end{cases}
\end{align*}
so that by Stokes' theorem
\begin{align*}\label{Stokes}
\int_{Q_{i,y}^+}\big(\partial_{\nu^\bot}u_1\big)\,\mathrm{d}x=\int_{\partial Q_{i,y}^+}u_1\big(\nu^\bot\cdot \nu_{\partial Q_{i,y}^+}\big)\,\mathrm{d}\H^1=\int_{\Gamma_i^h}u_1\nu_1\mathrm{d}\H^1-\int_{\Gamma_{i,y}^h}u_1\nu_1\mathrm{d}\H^1\,.
\end{align*}
In particular, the above identity, \eqref{eq: basic_contradiction_inequality_2}, \eqref{eq:Linfty_estimate}, and the fact that $\L^2(Q_{i,y}^+)\leq \sigma\in [2\gamma,4\gamma]$, yield
\begin{align*}
\bigg|\int_{\Gamma_{i,y}^h} u_1\,\mathrm{d}\H^1\bigg|&\leq \bigg|\int_{\Gamma_i^h} u_1\,\mathrm{d}\H^1\bigg|+\frac{1}{\nu_1}\bigg|\int_{Q_{i,y}^+}\partial_{\nu^\bot}u_1\bigg|\\
&\leq \L^1(\Gamma_i^h)\|u_1\|_{L^{\infty}(\Gamma_i^h)}+\frac{1}{\nu_1}[\L^2(Q_{i,y}^+)]^{\frac{1}{2}}\|\partial_{\nu^\bot}u_1\|_{L^2(B_1^+(\vec{z}_2))}
\\
&\lesssim\sqrt{c}\bigg(\frac{1}{\sqrt\gamma}+\frac{\sqrt{\gamma}}{\nu_1}\bigg)|G|\mu_\gamma
\,. 
\end{align*}
By the coarea formula, and since $\L^2(Q_i)\sim \sigma\sim\gamma$,
\begin{align*}
\big|\langle u_1\rangle_{Q_i}\big|&\sim \frac{1}{\gamma}\bigg|\int_{\Gamma_i} \nu_1\mathrm{d}\H^1(y) \int_{\Gamma_{i,y}^h} u_1\,\mathrm{d}\H^1\bigg|\leq \frac{1}{\gamma} \int_{\Gamma_i} \nu_1\mathrm{d}\H^1(y)\bigg|\int_{\Gamma_{i,y}^h} u_1\,\mathrm{d}\H^1\bigg| \\
&\lesssim \sqrt{c}\bigg(\frac{1}{\sqrt\gamma}+\frac{\sqrt{\gamma}}{\nu_1}\bigg)|G|\mu_\gamma\,,
\end{align*}
which proves  \eqref{small_averages}. 
In particular, by the triange inequality, 
\begin{equation}\label{eq:small_difference_of_averages}
\big|\langle u_1\rangle_{Q_3}-\langle u_1\rangle_{Q_1}\big|\lesssim 
\sqrt{c}\bigg(\frac{1}{\sqrt\gamma}+\frac{\sqrt{\gamma}}{\nu_1}\bigg)|G|\mu_\gamma\,.
\end{equation}
\end{step}

\begin{step}{\bf (2c)} \underline{(Contradiction argument)} In this step we arrive at a contradiction to \eqref{eq:small_difference_of_averages} similarly to Step 5 in the proof of \cite[Lemma 3.2]{knupfer2011minimal}. Writing
\[\langle u_1 \rangle_{Q_1}=\fint_{[x_1,\, x_2]}\mathrm{d}x\bigg( \frac{\nu_1}{\sigma}\int_{\Gamma_x} u_1\,\mathrm{d}\H^1\bigg)\,,\]
where $\Gamma_x:=(x+\R\nu^\bot)\cap S_\sigma$, there exists $x^*\in [x_1,\, x_2]$ such that
\begin{equation}\label{average_slice}
\int_{\Gamma_{x^*}}u_1\,\mathrm{d}\H^1=\frac{\sigma}{\nu_1} \langle u_1 \rangle_{Q_1}\,.
\end{equation}
In the \textit{tilted} $(e_1,\nu^\bot)$-coordinates and for every $x\in [x_3,\,x_4]$, by the fundamental theorem of calculus we have
\[u_1(x,y)=u_1(x^*,y)+\int_{[x^*,\,x]}\partial_1u_1\,\mathrm{d}\H^1\,,\]
and therefore by \eqref{average_slice},
\begin{align}\label{eq: fundamental_theorem_calculus_1}
\langle u_1\rangle_{Q_3}&=\fint_{[x_3,\,x_4]}\,\mathrm{d}x\Big(\frac{\nu_1}{\sigma}\int_{\Gamma_x} u_1(x,y)\,\mathrm{d}y \Big) \nonumber\\
&=\fint_{[x_3,\,x_4]}\,\mathrm{d}x\Big(\frac{\nu_1}{\sigma}\int_{\Gamma_x}\bigg(u_1(x^*,y)+\int_{[x^*,\,x]}\partial_1u_1\,\mathrm{d}\H^1\bigg)\,\mathrm{d}y \Big) \nonumber\\[3pt]
&=
\langle u_1 \rangle_{Q_1} 
+\frac{\nu_1}{\sigma}\fint_{[x_3,\,x_4]}\int_{\Gamma_x}\int_{[x^*,\,x]}(\partial_1u_1-\chi G_{11})\,\mathrm{d}\H^1\,\mathrm{d}y \,\mathrm{d}x \nonumber\\[3pt]
&\qquad \quad \qquad \qquad \qquad \qquad \quad +\frac{\nu_1}{\sigma}G_{11}\fint_{[x_3,\,x_4]}\int_{\Gamma_x}\int_{[x^*,\,x]}\chi\,\mathrm{d}\H^1\,\mathrm{d}y\,\mathrm{d}x\,.
\end{align}
To complete the argument we assume that $|a_1|=\max\{|a_1|,|a_2|\}\geq \frac{|G|}{\sqrt{2}}>0$,
since otherwise 
exactly the same argument can be repeated with $(u_2, G_{21})$ in the place of $(u_1, G_{11})$. 
By \eqref{eq: fundamental_theorem_calculus_1}, the fact that $B_{\gamma}^+(\vec{z}_2)\subset Q_2$, again the coarea formula, 
\eqref{eq: basic_contradiction_inequality_1}, 
and that $\L^2(Q)\lesssim \sigma\in[2\gamma,4\gamma]$, we arrive at
\begin{align*}\label{contradiction_in_2_d}
\bigg(\frac{|G|\nu_1}{\sqrt{2}}\bigg)\mu_\gamma&\leq |G_{11}|\int_{B^+_{\gamma}(\vec{z}_2)}\chi \lesssim 
|G_{11}|\fint_{[x_3,\,x_4]}\nu_1\int_{\Gamma_x}\int_{[x^*,\,x]}\chi\,\mathrm{d}\H^1\,\mathrm{d}y\,\mathrm{d}x\\[3pt]
&\lesssim \sigma 
\big|\langle u_1\rangle_{Q_3}-\langle u_1\rangle_{Q_1}\big|+\int_{Q}|\nabla u-\chi G|
\\
&\lesssim \sigma \big|\langle u_1\rangle_{Q_3}-\langle u_1\rangle_{Q_1}\big|+[\L^2(Q)]^{\frac{1}{2}}\,\|\nabla u-\chi G\|_{L^2(B_1^+(\vec{z}_2))}
\\
&\lesssim \sqrt{c\gamma}|G|\mu_\gamma\,.
\end{align*}
By the choice \eqref{eq: choice_of_theta_gamma} (\textit{i.e.}, that $\gamma\sim \nu_1$), the last inequality leads to a contradiction, as long as the final constant $c>0$ is chosen small enough with respect to $\nu_1$, so that 
\begin{equation}\label{eq:final_choice_of_gamma_2d}
\sqrt{c\gamma}\ll \nu_1\implies c\ll \nu_1\,.
\end{equation}
For instance we can choose $c=c_2\nu_1$, for a sufficiently small absolute constant $c_2\in (0,1)$, which proves the claim for $d=2$, also in accordance with \eqref{lower_bound_from_knupfer_kohn}. \\[-3pt]
\end{step}

\begin{step}{\bf (3)} \underline{(Proof in higher dimensions $d\geq 3$ by induction)}
We follow again the lines of the proof of \cite[Lemma 3.4]{knupfer2011minimal}.
In our setting however, we can avoid the averaging argument therein, making the proof somehow uniform in all dimensions. This is a consequence of the absence of (linear) gauge invariance in this model, together with the use of the appropriately \textit{tilted} cages, that encode the role of the direction of the rank-$1$ connection.

Analogously to Steps 1 and 2 in the proof of \cite[Lemma 3.4]{knupfer2011minimal}, by means of an inductive argument we can ensure that there is essentially only majority phase concentrated on the upper horizontal boundary of the higher dimensional cages. To be more specific, having shown the proposition for $d=2$ in the previous step, for the inductive hypothesis assume that it also holds for all $d'\in \N$ with $2\leq d'\leq d-1$. 

Let again $z:=z_de_d$, $0\leq z_d \leq \theta_{d,\nu}$, and $c_{d,\nu}, \gamma_{d,\nu}>0$ denote the corresponding constants for which the proposition will be valid in dimension $d$, that will be chosen suitably small. By contradiction, we again assume that there exists a $\chi$ such that although
\begin{equation}\label{contradiction_for_chi_higher_d}
\|\chi\|_{L^1(B_1^+(z))}\leq c_{d,\nu}\, \ \ \text{and }\ |D\chi|(B_1^+(z))\leq  c_{d,\nu}\,,
\end{equation}   
there exists $u\in H^1(\R^d_+;\R^d)$ such that 
\begin{equation}\label{contradiction_for_u_higher_d}
\int_{B_1^+(z)}|\nabla u-\chi G|^2\,\mathrm{d}x< c_{d,\nu}|G|\mu^2_{\gamma_{d,\nu}}\,, \text{where } \mu_{\gamma_{d,\nu}}:=\|\chi\|_{L^1(B^+_{\gamma_{d,\nu}}(z))}\,.
\end{equation}   
Denoting $\nu^\bot:=-\nu_de_1+\nu_1e_d$ (assuming again without restriction that $\nu_d\leq 0<\nu_1$), \eqref{contradiction_for_u_higher_d} implies as before that
\begin{equation}\label{contradiction_for_directional_derivative_higher_d}
\int_{B_1^+(z)}|\partial_{\nu^ \bot}u|^2\,\mathrm{d}x< c_{d,\nu}|G|^2\mu^2_{\gamma_{d,\nu}}\,.
\end{equation}   
As in Step (2a) (cf. also Appendix \ref{sec: choice_of_cages}), by means of a simple geometric argument and Fubini's theorem, we can again find $\sigma\in [2\gamma_{d,\nu},4\gamma_{d,\nu}]$ and adjacent \textit{tilted} cages $Q_1,Q_2,Q_3$ (with $Q:=\bigcup_{i=1}^3Q_i$) such that
\begin{align}\label{eq: properties_of_higher_d_cages}
\begin{split}
\rm{(i)}&\quad \nu_{\partial Q_i}\in \{\pm \nu, \pm e_2, \dots, \pm e_{d}\} \ \ \forall i=1,2,3\,,\\[2pt]
\rm{(ii)}& \quad \text {The faces of } Q_i \text{ with outer normal } -e_d \text{ lie on the hyperplane } \{x_d=0\}\,,\\[2pt]
\rm{(iii)}& \quad Q_i \text{ have height } \sigma \text{ in the } e_d\text{-direction, and  all its edges are of order }1 \,,\\[2pt]
\rm{(iv)}& \quad B^+_{\gamma_{d,\nu}}(z)\subset Q_2\subset Q\subset\subset B_1^+(z)\,,\\[2pt]
\rm{(v)}&\quad \|\chi\|_{L^1(\Pi_\sigma)}+|D\chi|(\Pi_\sigma)<\frac{c_{d,\nu}}{\gamma_{d,\nu}}\,,\\[2pt]
\rm{(vi)}&\quad \int_{\Pi_\sigma}|\nabla u-\chi G|^2\,\mathrm{d}\H^{d-1}<\frac{c_{d,\nu}}{\gamma_{d,\nu}}|G|^2\mu^2_{\gamma_{d,\nu}}\,.
\end{split}
\end{align}
Here, we have denoted
\begin{equation}\label{eq: Pi_sigma}
\Pi_\sigma:=B_1^+(z)\cap \{x\in \R^d\colon x_d=\sigma\}=B^{d-1}_{\tilde \rho}\times\{\sigma e_d\}\,,
\end{equation}
for some $\tilde \rho\sim 1$ (depending on $0<\sigma\ll 1$). 
Let us also denote by $T_\sigma$ the face of $Q$ with outer unit normal $e_d$, at vertical height $\sigma$. Note that by choosing $\gamma_{d,\nu}\in (0,1)$ sufficiently small, we can ensure that
\begin{equation}\label{eq: upper_phase_inclusion}
T_\sigma\subset B^{d-1}_{\alpha_{d-1}\tilde \rho}\times\{\sigma e_d\}\subset B^{d-1}_{\tilde \rho}\times\{\sigma e_d\}=\Pi_\sigma\,,
\end{equation}
where $\alpha_{d-1}>0$ is the constant of \cite[Claim 3.1]{Tribuzio-Rueland_1} (in dimension $d-1$).\\[3pt]
\end{step}
\begin{step}{\bf (3a)} \underline{($\chi$ is $\H^{d-1}$-negligible on $\Pi_\sigma$)} Although in this higher dimensional setting we cannot deduce by the previous arguments that $\chi\equiv 0$ $\H^{d-1}$-a.e. on $T_\sigma$, as was the case for $d=2$, we show in this step that we can still have the estimate of the form
\begin{equation*}
\int_{T_\sigma}\chi\,\mathrm{d}\H^{d-1}\ll \mu_{\gamma_{d,\nu}}\,.
\end{equation*}
Recalling that $G=a\otimes \nu=(\nu_1 a)\otimes e_1+(\nu_d a)\otimes e_d$, there exists $i_0\in \{1,\dots,d\}$ for which
\begin{equation}\label{eq: choice_of_max_i}
|a_{i_0}|\geq \frac{|a|}{\sqrt d}=\frac{|G|}{\sqrt d}>0\,.
\end{equation}	
Picking an index $l\in \{1,\dots,d\}$ with $l\neq i_0$, and recalling the notation in \eqref{eq:_compatible_strains}, we consider the reduced matrix $\tilde G\in \mathcal{R}_1(d-1)$, which is obtained from $G$ by deleting the $l$-th row and the $d$-th column. Obsiously $\tilde G$ is still a rank-1 matrix, since $|a_{i_0}|,\nu_1>0$. 

We further consider the map $\tilde u^\sigma\in H^1(\R^{d-1}; \R^{d-1})$, defined via
\[\tilde u^{\sigma}(x_1,\dots,x_{d-1}):=(u_j(x_1,\dots,x_{d-1},\sigma))_{j\neq l}\,,\] 
omitting from $u$ its $l$-th component. If $\nabla'$ denotes the restriction of the gradient operator in the $\{e_1,\dots,e_{d-1}\}$-directions, we can easily deduce the following. By choosing the constant $c_{d,\nu}, \gamma_{d,\nu} >0$ suitable with respect to the constant $c:=c(d-1)>0$ of \cite[Claim 3.1]{Tribuzio-Rueland_1}, and since $\tilde \rho\sim 1$, by \eqref{eq: properties_of_higher_d_cages} and \eqref{eq: Pi_sigma} we have
\begin{equation*}
\int_{B^{d-1}_{\tilde \rho}\times\{\sigma\}}\chi \,\mathrm{d}\H^{d-1}<\frac{c_{d,\nu}}{\gamma_{d,\nu}}\leq c\tilde \rho^d \ \text{and } \ |D\chi|(B^{d-1}_{\tilde \rho}\times\{\sigma\})<\frac{c_{d,\nu}}{\gamma_{d,\nu}} \leq c\tilde\rho^{d-1}\,,
\end{equation*}
i.e., \cite[Claim 3.1]{Tribuzio-Rueland_1} is applicable in dimension $(d-1)$ in the ball $B^{d-1}_{\tilde \rho}$. In particular,
\begin{equation}\label{eq: lower_bound_in_d-1}
\int_{B^{d-1}_{\tilde \rho }}|\nabla'\tilde u^{\sigma}-\chi(\cdot,\sigma) \tilde G|^2\,\mathrm{d}\H^{d-1}\geq c\tilde \rho^{1-d}|\tilde G|^2\|\chi \|^2_{L^1\big(B^{d-1}_{\alpha_{d-1}\tilde \rho}\times\{\sigma\}\big)}\,.
\end{equation}
Combining \eqref{eq: properties_of_higher_d_cages}(vi), \eqref{eq: upper_phase_inclusion}, \eqref{eq: choice_of_max_i}, \eqref{eq: lower_bound_in_d-1}, and the fact that $|\tilde G|\geq |a_{i_0}|\nu_1$, we infer
\begin{align*}
c\tilde \rho^{1-d}\frac{|G|^2\nu^2_1}{d}\|\chi\|_{L^1(T_\sigma)}^2&\leq c\tilde \rho^{1-d}|\tilde G|^2\|\chi\|_{L^1(B^{d-1}_{\alpha_{d-1}\tilde \rho}\times\{\sigma\})}^2\leq \int_{B^{d-1}_{\tilde \rho }}|\nabla'\tilde u^{\sigma}-\chi(\cdot,\sigma) \tilde G|^2\,\mathrm{d}\H^{d-1}\\[2pt]
&\leq  \int_{\Pi_\sigma}|\nabla u-\chi G|^2\,\mathrm{d}\H^{d-1}<\frac{c_{d,\nu}}{\gamma_{d,\nu}}|G|^2\mu_{\gamma_{d,\nu}}^2\,,
\end{align*} 
\textit{i.e.},
\begin{equation}\label{eq: smallness_of_chi_on_upper_boundary}
\int_{T_\sigma}\chi\,\mathrm{d}\H^{d-1}\lesssim \frac{1}{\nu_1}\sqrt{\frac{c_{d,\nu}}{\gamma_{d,\nu}}}\mu_{\gamma_{d,\nu}}\,,
\end{equation}
which proves the desired assertion.
\\[4pt]
\end{step}	
\begin{step}{\bf (3b)} \underline{(Contradiction as in the $2$-dimensional case)} Recalling the choice of the index $i_0$ in \eqref{eq: choice_of_max_i}, and by the translation invariance of the problem, we can assume without restriction that 
\[\langle u_{i_0} \rangle_{T_\sigma}=0\,.\] 
	
Since by construction $T_\sigma$ is a $(d-1)$-dimensional normal cuboid with side lengths all of order 1, by the $L^1$-Poincar{\' e} inequality on $T_\sigma$, \eqref{eq: properties_of_higher_d_cages}(vi) and \eqref{eq: smallness_of_chi_on_upper_boundary} we deduce that
\begin{align}\label{u_i_o_small_in_L_1}
\begin{split}	
\|u_{i_0}\|_{L^1(T_\sigma)}&\lesssim \|\nabla'u_{i_0}\|_{L^1(T_\sigma)}\leq \|\nabla u\|_{L^1(T_\sigma)}\leq \|\nabla u-\chi G\|_{L^1(T_\sigma)}+|G|\|\chi\|_{L^1(T_\sigma)}\\[3pt]
&\lesssim \|\nabla u-\chi G\|_{L^2(\Pi_\sigma)}+|G|\|\chi\|_{L^1(T_\sigma)}\\
& \lesssim\sqrt{\frac{c_{d,\nu}}{\gamma_{d,\nu}}}\big(1+\frac{1}{\nu_1}\big)|G|\mu_{\gamma_{d,\nu}}
\,,
\end{split}
\end{align}  
With \eqref{u_i_o_small_in_L_1} at hand, the rest of the contradiction argument can be performed as in the 2-dimensional case (see Step 2). Indeed, for every $i=1,2,3$ and $s\in (0,\sigma]$, let 
\[T^h_{i,s}:=Q_i\cap \{x_d=s\}\, \ \text{and } Q_{i,s}^+:=\mathrm{conv}(T_{i,s}^h\cup T^h_{i,\sigma})\,,\]
so that $Q_i=\bigcup_{s\in (0,\sigma)}T^h_{i,s}$. Note that $\L^d(Q_{i,s}^+)\lesssim \sigma\in[2\gamma_{d,\nu},4\gamma_{d,\nu}]$. Since
\[\nu_{\partial Q_{i,s}^+}\in\{\pm \nu, \pm e_2,\dots,\pm e_d\}\, \ \text {and } \ \nu^{\bot}\cdot \nu=\nu^{\bot}\cdot e_2=\dots \nu^{\bot}\cdot e_{d-1}=0\,, \]
by using Stokes' theorem in $Q_{i,s}^+$, \eqref{contradiction_for_directional_derivative_higher_d} and \eqref{u_i_o_small_in_L_1}, (as before \eqref{eq:small_difference_of_averages}) we infer again that 
\begin{align*}
\bigg|\int_{T_{i,s}^h} u_{i_0}\,\mathrm{d}\H^{d-1}\bigg|&\leq \bigg|\int_{T^h_{i,\sigma}} u_{i_0}\,\mathrm{d}\H^{d-1}\bigg|+\frac{1}{\nu_1}\bigg|\int_{Q_{i,s}^+}\partial_{\nu^\bot}u_{i_0}\,\mathrm{d}x\bigg|\\
&\leq \|u_{i_0}\|_{L^1(T^h_{i,\sigma})}+\frac{1}{\nu_1}[\L^d(Q_{i,s}^+)]^{\frac{1}{2}}\|\partial_{\nu^\bot}u_{i_0}\|_{L^2(B_1^+(z))}
\\
&\lesssim \frac{\sqrt{c_{d,\nu}}}{\nu_1}\Big(\frac{1}{\sqrt{\gamma_{d,\nu}}}+\sqrt{\gamma_{d,\nu}}\Big) |G|\mu_{\gamma_{d,\nu}}\,.
\end{align*}
By the coarea formula and the fact that $\L^d(Q_{i})\sim \sigma\sim\gamma_{d,\nu}$, we get similarly to the 2-dimensional case,
\begin{align*}
\big|\langle u_{i_0}\rangle_{Q_i}\big|&\lesssim \frac{\sqrt{c_{d,\nu}}}{\gamma_{d,\nu}}\Big(\frac{1}{\sqrt{\gamma_{d,\nu}}}+\sqrt{\gamma_{d,\nu}}\Big) |G|\mu_{\gamma_{d,\nu}}
\end{align*}
in particular,
\begin{equation}\label{eq:higher_d_small_difference_of_averages}
\big|\langle u_{i_0}\rangle_{Q_3}-\langle u_{i_0}\rangle_{Q_1}\big|\lesssim  \frac{\sqrt{c_{d,\nu}}}{\gamma_{d,\nu}}\Big(\frac{1}{\sqrt{\gamma_{d,\nu}}}+\sqrt{\gamma_{d,\nu}}\Big) |G|\mu_{\gamma_{d,\nu}}\,.
\end{equation}
\end{step}
The estimate \eqref{eq:higher_d_small_difference_of_averages} and \eqref{eq: properties_of_higher_d_cages}(iv) give the final desired contradiction, by slicing in the $e_1$-direction, exactly as in Step (2c), by first picking a \textit{tilted} slice (with normal $\pm \nu$) in $Q_1$, on which the average of $u_{i_0}$ equals $\langle u_{i_0}\rangle_{Q_1}$, using the fundamental theorem of calculus on all horizontal lines passing through this and all the corresponding tilted slices of $Q_3$, and using the fact that $|G_{1i_0}|=\nu_1|a_{i_0}|\geq \frac{\nu_1|G|}{\sqrt{d}}>0$. Then by the very same argument as above \eqref{eq:final_choice_of_gamma_2d}, using \eqref{contradiction_for_u_higher_d}, \eqref{eq:higher_d_small_difference_of_averages}, and that $\L^d(Q)\sim \sigma\sim\gamma_{d,\nu}$, we would get 
\begin{align*}\label{contradiction_in_any_d}
\bigg(\frac{|G|\nu_1}{\sqrt{d}}\bigg)\mu_\gamma&\lesssim \sigma \big|\langle u_1\rangle_{Q_3}-\langle u_1\rangle_{Q_1}\big|+[\L^2(Q)]^{\frac{1}{2}}\,\|\nabla u-\chi G\|_{L^2(B_1^+(\vec{z}_2))}
\bigg)\\
& \lesssim \sqrt{c_{d,\nu}}\bigg(\frac{1}{\sqrt{\gamma_{d,\nu}}}+\sqrt{\gamma_{d,\nu}}\bigg)|G|\mu_{\gamma_{d,\nu}}\\
&\lesssim \sqrt{\frac{c_{d,\nu}}{\gamma_{d,\nu}}}|G|\mu_{\gamma_{d,\nu}}\,.
\end{align*}
Choosing again $\gamma_{d,\nu}\sim \nu_1$, the last inequality leads to a contradiction, as long as the final constant $c_{d,\nu}>0$ is chosen small enough with respect to $\nu_1$, so that 
\begin{equation}\label{eq:final_choice_of_gamma_any_d}
\sqrt{\frac{c_{d,\nu}}{\gamma_{d,\nu}}}\ll \nu_1\implies c_{d,\nu}\ll \nu_1^3\,,
\end{equation}
\textit{e.g.}, we can choose $c_{d,\nu}=c_d\nu_1^3$, for a sufficiently small dimensional constant $c_d\in (0,1)$. \\
This completes the proof in all dimensions.
\end{proof}

\begin{remark}\label{rem:power_law_lower}
\normalfont Our method of proof provides also an explicit dependence of the constant $c_{d,\nu}$ with respect to $\dist(\nu,\{\pm e_d\})$. In particular, for $\nu=\nu_1e_1+\nu_d e_d$ with $\nu_1>0$ (something that as we saw can be assumed without restriction), and since $\nu_1\sim\dist(\nu,\{\pm e_d\})$, the choices below \eqref{eq:final_choice_of_gamma_2d} and \eqref{eq:final_choice_of_gamma_any_d}, yield
\begin{equation}\label{eq:explicit_c_d_nu}
c=c_{d,\nu}
\sim
\begin{cases}
\dist(\nu,\{\pm e_d\}) &\text{if } d=2\,,\\[2pt]
\dist^3(\nu,\{\pm e_d\}) &\text{if } d\geq 3\,,
\end{cases}
\end{equation}
up to a sufficiently small dimensional constant. The difference in scaling comes essentially from the fact that for $d\geq 3$ we can only ensure \eqref{eq: smallness_of_chi_on_upper_boundary}, in contrast to the case $d=2$, where the discreteness of $D\chi$ on lines guarantess the even stronger property \eqref{properties_of_tilting}(iii). 
\end{remark}

\section{Proof of Theorem \ref{main_theorem_rescaled}$(i)$} \label{sec:4}
\subsection{The lower bound}
The proof of the lower bound follows the lines of proof for the lower bound in \cite[Theorem 3.6]{knupfer2011minimal}, by slightly adapting the covering argument therein to the case of the upper half-space. For the sake of making the paper self-contained, but also in order to keep track of the specific dependence of certain constants with respect to the rank-$1$ direction $\nu$, we include the proof by addressing the necessary adaptations.
For this purpose, for $\mu\leq |G|^{-2d}$, by the (relative) isoperimetric inequality in half-space, we have that 
\[\mathcal{E}(u,\chi)\geq |D\chi|(\R^d_+)\gtrsim \|\chi\|_{L^1(\R^d_+)}^{\frac{d-1}{d}}=\mu^{\frac{d-1}{d}}\ \ \forall (u,\chi)\in \mathcal{A}(\mu)\ \ \implies E(\mu)\gtrsim  \mu^{\frac{d-1}{d}}\,,\]
hence it suffices to consider the case $\mu\geq |G|^{-2d}$. In that respect, let 
\begin{equation}\label{eq: M_chi_def}
M_\chi := \mathrm{spt} (\chi)\subset \{x\in \R^d\colon x_d\geq 0\}\,,  \ \text{so that } \L^d(M_\chi)=\mu\,.
\end{equation} 
For $\L^d$-a.e. $x\in M_\chi$ we have 
\[\Theta^d(M_\chi,x):=\lim_{\rho\searrow 0}\frac{\L^d(M_\chi\cap B_\rho(x))}{\omega_d\rho^d}=1\,,\]
a property which without restriction we assume that holds true for every  $x\in M_\chi\cap \R^d_+$, since this will be enough for the covering argument. For every such $x\in M_\chi\cap \R^d_+$, we define
\begin{equation}\label{eq: maximal_radius}
r_x := \inf \Big\{r>0\colon 
r^{-d}\L^d(M_\chi\cap B_r(x)) \leq \tau_{\nu} \min\big\{1, |G|^{-\frac{2d}{2d-1}}\L^d(M_\chi\cap B_r(x))^{\frac{-1}{2d-1}}\big\}\Big\}\,,
\end{equation}
where $\tau_{\nu}\in (0,1)$ is a sufficiently small 
constant to be suitably chosen later. (For the justification of the existence of such $r_x>0$ and a motivation for its introduction, cf. the paragraph below \cite[Equation (3.42)]{knupfer2011minimal}). By the minimality condition that $r_x$ satisfies, it is immediate that 

\begin{align}\label{eq: cases_for_R_x}
\begin{split}
\text{either } \quad  \L^d\big(M_\chi\cap B_{r_x}(x)\big)\leq |G|^{-2d} & \implies \L^d\big(M_\chi\cap B_{r_x}(x)\big)= \tau_{\nu} r_x^d\,,\\[2pt]
\text{or } \quad \L^d\big(M_\chi\cap B_{r_x}(x)\big)\geq |G|^{-2d} & \implies \L^d\big(M_\chi\cap B_{r_x}(x)\big)^{\frac{2d}{2d-1}}=\tau_{\nu} |G|^{-\frac{2d}{2d-1}}r_x^d\,.
\end{split}
\end{align}
\begin{step}{1} \underline{(The covering argument)}
Starting from the trivial covering 
\[M_\chi\subset \bigcup_{x\in M_\chi} B_{r_x/5}(x)\,, \ \text{with} \ \sup_{x\in M_\chi} r_x\le \tau_{\nu}^{-\frac{1}{d}}(|G|\mu)^{\frac{2}{2d-1}}<+\infty\,, 
\] 
(the last bounds following from \eqref{eq: cases_for_R_x} and $\mu\geq |G|^{-2d}$, independently of $x\in M_\chi$), by Vitali’s covering lemma there exists an at most countable subset of points $(z_i)_{i\in \N} \subset M_\chi$ such that
\begin{equation}\label{Vitali's_covering}
M_\chi\subset \bigcup_{i=1}^\infty B_{r_i}(z_i)\,, \  (z_i,r_i):=(z_{x_i}, r_{x_i}) \ \ \text{and } \{B_{r_i/5}(z_i)\}_{i\in \N} \ \text {are pairwise disjoint}\,.
\end{equation}
By \eqref{Vitali's_covering} we have
\begin{equation}\label{eq: additivity_of_volume}
\sum_{i=1}^\infty \L^d(M_\chi\cap B_{r_i}(z_i))\geq \L^d(M_\chi)=\mu\,, 
\end{equation}
and for every pair $(u,\chi)\in \mathcal{A}(\mu)$ that we fix next:
\begin{equation}\label{additivity_of_energy}
\sum_{i=1}^\infty \E\big((u,\chi); B^+_{r_i/5}(z_i)\big)\leq \E(u,\chi)\,.
\end{equation}
Exactly as in \cite{knupfer2011minimal} (cf. the proof of (3.44) and (3.45) therein) it can be checked that
\begin{equation}\label{inverse_volume_bounds}
\L^d(M_\chi\cap B_{\gamma r_i/5}(z_i))\geq c_{1,d,\nu}\ \L^d(M_\chi \cap B_{r_i}(z_i))\,,
\end{equation}
where 
\begin{equation}\label{eq: c_1_2_d_nu_explicit}
c_{1,d,\nu}\sim \gamma^d\,,
\end{equation}
and $\gamma=\gamma(d,\nu)\in (0,1)$ is the constant of Proposition \ref{elastic_lower_bound_on_half_balls}. Moreover, 
\begin{equation}\label{inverse_energy_bounds}
\E\big((u,\chi); B^+_{r_i/5}(z_i)\big)\geq c_{2,d,\nu}\ 
|G|^{\frac{2d-2}{2d-1}}\big[\L^d(M_\chi\cap B^+_{\gamma r_i/5}(z_i))\big]^{\frac{2d-2}{2d-1}}\, \quad \forall i\in \N\,,
\end{equation}
where 
\begin{equation}\label{eq: c_2_d_nu_explicit}
c_{2,d,\nu} \sim c^2\gamma^\frac{2d^2}{2d-1}
\,.
\end{equation}
Once \eqref{inverse_volume_bounds} and \eqref{inverse_energy_bounds} are established (which we do at the end of the proof),
for the lower bound we easily estimate using \eqref{Vitali's_covering}, \eqref{inverse_energy_bounds}, \eqref{inverse_volume_bounds} and \eqref{eq: additivity_of_volume},
\begin{align}\label{eq_lower_bound_pf}
\E(u,\chi)&\geq \sum_{i=1}^\infty \E((u,\chi); B^+_{r_i/5}(z_i)) \geq c_{2,d,\nu} \ 
|G|^{\frac{2d-2}{2d-1}}\sum_{i=1}^\infty\big[\L^d(M_\chi\cap B^+_{\gamma r_i/5}(z_i))\big]^{\frac{2d-2}{2d-1}} \nonumber \\[3pt]
&\geq c_{2,d,\nu} \ 
|G|^{\frac{2d-2}{2d-1}}\bigg[\sum_{i=1}^\infty\L^d\big(M_\chi\cap B_{\gamma r_i/5}(z_i)\big)\bigg]^{\frac{2d-2}{2d-1}}\nonumber\\
&\geq c_{2,d,\nu}(c_{1,d,\nu})^{\frac{2d-2}{2d-1}} \ 
|G|^{\frac{2d-2}{2d-1}}\bigg[\sum_{i=1}^\infty\L^d\big(M_\chi\cap B_{r_i}(z_i)\big)\bigg]^{\frac{2d-2}{2d-1}} \nonumber\\[3pt]
& \geq c_{2,d,\nu}\big(c_{1,d,\nu}\big)^{\frac{2d-2}{2d-1}} \ 
|G|^{\frac{2d-2}{2d-1}}\mu^{\frac{2d-2}{2d-1}} 
\,,
\end{align} 
where we additionally used \eqref{eq: M_chi_def} and the algebraic inequality
\[\sum_{i=1}^\infty t_i^\alpha\geq \bigg(\sum_{i=1}^\infty t_i\bigg)^\alpha \ \ \forall t_i\geq 0\, \ \text{and } 0<\alpha<1 \ \ \ \text{(which can be proved by induction)}\,.\] By \eqref{eq: rescaled_energy} and \eqref{eq: rescaled_minimal_energy}, \eqref{eq_lower_bound_pf} yields the lower bound in the second case of \eqref{eq: rescaled_generic_energy_scaling}. 

The verification of \eqref{inverse_volume_bounds} and \eqref{inverse_energy_bounds} is done by following verbatim the corresponding arguments in \cite{knupfer2011minimal}. For completeness, and in order to keep track of the dependence of the relevant constants on $\nu$, we give them in detail also here.
\end{step}

\begin{step}{2} \underline{(Proof of \eqref{inverse_volume_bounds})}
In order to prove \eqref{inverse_volume_bounds} 
we use the minimality condition for $r_i$ with respect to \eqref{eq: maximal_radius}. In particular, if $i\in \N$ is such that
\[\L^d(M_\chi\cap B_{r_i}(z_i))\leq |G|^{-2d}\,, \]
then, since $\gamma r_i/5<r_i$, by \eqref{eq: maximal_radius} and \eqref{eq: cases_for_R_x} we have
\begin{align}\label{case_1_inv_vol}
&\frac{\L^d(M_\chi\cap B_{\gamma r_i/5}(z_i))}{(\gamma r_i/5)^d}\geq \tau_{\nu}= \frac{\L^d(M_\chi\cap B_{r_i}(z_i))}{r_i^d} \nonumber \\[2pt]
\implies& \L^d(M_\chi\cap B_{\gamma r_i/5}(z_i))\geq (\gamma/5)^d\L^d(M_\chi\cap B_{r_i}(z_i))\,.
\end{align}
If $i\in \N$ is such that
\[\L^d(M_\chi\cap B_{r_i}(z_i))> |G|^{-2d}\ \ \text{ but } \ \L^d(M_\chi\cap B_{\gamma r_i/5}(z_i))\leq |G|^{-2d}\,,\]
then, again since $\gamma r_i/5<r_i$, by \eqref{eq: maximal_radius} and \eqref{eq: cases_for_R_x}, we have
\begin{align}\label{case_2_inv_vol}
&\frac{\L^d(M_\chi\cap B_{\gamma r_i/5}(z_i))}{(\gamma r_i/5)^d}\geq \tau_{ \nu}=|G|^\frac{2d}{2d-1}\frac{\big(\L^d(M_\chi\cap B_{r_i}(z_i))\big)^{\frac{2d}{2d-1}}}{r_i^d}
 \nonumber \\[3pt]
\implies& \L^d(M_\chi\cap B_{\gamma r_i/5}(z_i))\geq (\gamma/5)^d \big(|G|^{2d}\L^d(M_\chi\cap B_{r_i}(z_i))\big)^{\frac{1}{2d-1}}\L^d(M_\chi\cap B_{r_i}(z_i))\nonumber \\[3pt]
\implies& \L^d(M_\chi\cap B_{\gamma r_i/5}(z_i))> (\gamma/5)^d\L^d(M_\chi\cap B_{r_i}(z_i))\,.
\end{align} 
Finally, if $i\in \N$ is such that
\[\L^d(M_\chi\cap B_{r_i}(z_i))> |G|^{-2d}\ \text{ and also } \L^d(M_\chi\cap B_{\gamma r_i/5}(z_i))> |G|^{-2d}\,, \]
then, as before
\begin{align}\label{case_3_inv_vol}
&\frac{\L^d(M_\chi\cap B_{\gamma r_i/5}(z_i))}{(\gamma r_i/5)^d}> (\tau_{\nu}|G|^{-\frac{2d}{2d-1}})\L^d(M_\chi\cap B_{\gamma r_i/5}(z_i))^{\frac{-1}{2d-1}}
 \nonumber \\[2pt]
& \qquad \qquad \qquad \qquad \quad=\bigg(\frac{[\L^d(M_\chi\cap B_{r_i}(z_i))]^{\frac{2d}{2d-1}}}{r_i^d}\bigg)\L^d(M_\chi\cap B_{\gamma r_i/5}(z_i))^{\frac{-1}{2d-1}} \nonumber\\[2pt]
\implies& \L^d(M_\chi\cap B_{\gamma r_i/5}(z_i))> (\gamma/5)^{\frac{2d-1}{2}}\L^d(M_\chi\cap B_{r_i}(z_i))\,.
\end{align} 
Hence, since $\gamma:=\gamma(d,\nu)\in (0,1)$, by \eqref{case_1_inv_vol}-\eqref{case_3_inv_vol} we infer \eqref{inverse_volume_bounds} with the constant as in \eqref{eq: c_1_2_d_nu_explicit}.
\end{step}\\[2pt]

\begin{step}{3} \underline{(Proof of \eqref{inverse_energy_bounds})}
 If $i\in \N$ and the first case in \eqref{eq: cases_for_R_x} holds, then 
\begin{align*}
\begin{split}
\L^d(M_\chi\cap B^+_{r_i/5}(z_i))&\leq \L^d(M_\chi\cap B_{r_i}(z_i))=\tau_{\nu} r_i^d\\&<\frac{1}{4}\omega_d(r_i/5)^d\leq \frac{1}{2}\L^d(B^+_{r_i/5}(z_i))\,,
\end{split}
\end{align*}
\textit{i.e.}, 
\[\min\big\{\L^d(M_\chi\cap B^+_{r_i/5}(z_i))\,,\ \L^d(B^+_{r_i/5}(z_i)\setminus M_\chi)\big\}=\L^d(M_\chi\cap B^+_{r_i/5}(z_i))\,,\]
by choosing $\tau_{\nu}\in (0,1)$ small enough so that $\tau_{ \nu}<\frac{\omega_d}{4\cdot5^d}$. Therefore, by the relative isoperimetric inequality in the spherical cap $B^+_{r_i/5}(z_i)$ (cf. \cite[inequality (3.43)]{Ambrosio-Fusco-Pallara:2000}) we obtain
\[\H^{d-1}(\partial^*M_\chi\cap B^+_{r_i/5}(z_i))\geq C_d \big[\L^d(M_\chi\cap B^+_{r_i/5}(z_i))\big]^{d-1/d}\,.\]
The fact that the constant $C_d>0$ can be chosen purely dimensional can easily be  seen by the scaling invariance of the relative isoperimetric inequality, and the fact that $z_i\in \R^d_+$, which in particular implies that 
\[\{z\in \R^d\colon |z-z_i|\leq r_i/5\,,\ z_d\geq (z_i)_d\}\subset B^+_{r_i/5}(z_i)\subset B_{r_i}(z_i)\,.\]
(One can then take $C_d:=\min_{t\in [0,1/5]}C_{d,t}$, where $C_{d,t}>0$ is the relative isoperimetric constant with respect to the spherical cap $\{x\in \R^d\colon |x-te_d|\leq 1/5\,,x_d\geq 0\}$.) Therefore, using again \eqref{inverse_volume_bounds} and the first case in \eqref{eq: cases_for_R_x}, 

\begin{align}\label{eq: 1_pf_of_energy_lower_bound}
\E\big((u,\chi);B^+_{r_i/5}(z_i)\big)&\geq |D\chi|(B^+_{r_i/5}(z_i))=\H^{d-1}\big(\partial^*M_\chi\cap B^+_{r_i/5}(z_i)\big) \nonumber\\[3pt]
&\geq C_d \big[\L^d(M_\chi\cap B_{r_i/5}(z_i))\big]^{d-1/d}\gtrsim  (c_{1,d,\nu})^{\frac{d-1}{d}} \big[\L^d(M_\chi\cap B_{r_i}(z_i))\big]^{d-1/d}\nonumber\\[3pt]
&\gtrsim (c_{1,d,\nu})^{\frac{d-1}{d}} \big[\L^d(M_\chi\cap B_{r_i}(z_i))\big]^{\frac{-(d-1)}{d(2d-1)}}\big[\L^d(M_\chi\cap B_{r_i/5}(z_i))\big]^{\frac{2d-2}{2d-1}} \nonumber\\[3pt]
&\gtrsim (c_{1,d,\nu})^{\frac{d-1}{d}} |G|^{\frac{2d-2}{2d-1}}\big[\L^d(M_\chi\cap B^+_{\gamma r_i/5}(z_i))\big]^{\frac{2d-2}{2d-1}}\,.
\end{align}

Suppose now that $i\in \N$ is such that the second case in \eqref{eq: cases_for_R_x} holds, which in particular implies that
\[\tau_{\nu}^{1/2d}r_i^{1/2}\geq |G|^{-1}\,.\]
If \[|D\chi|\big(B_{r_i/5}^+(z_i)\big)\geq c (r_i/5)^{d-1}\,,\]
where $c>0$ is the constant from Proposition \ref{elastic_lower_bound_on_half_balls} (cf. also \eqref{eq:explicit_c_d_nu}). Then by \eqref{eq: cases_for_R_x}, we have
\begin{align}\label{eq: 2_pf_of_energy_lower_bound}
\begin{split}
\E\big((u,\chi);B^+_{r_i/5}(z_i)\big)&\geq |D\chi|(B^+_{r_i/5}(z_i))\geq c (r_i/5)^{d-1} 
\\[2pt]
&\sim c\tau_{\nu}^{-\frac{d-1}{d}}|G|^{\frac{2d-2}{2d-1}}\big[\L^d(M_\chi\cap B_{r_i}(z_i))\big]^{\frac{2d-2}{2d-1}}\\[2pt]
&\gtrsim c\tau_{\nu}^{-\frac{d-1}{d}}|G|^{\frac{2d-2}{2d-1}}\big[\L^d(M_\chi\cap B^+_{\gamma r_i/5}(z_i))\big]^{\frac{2d-2}{2d-1}}\,,
\end{split}
\end{align}
so we are left with the subcase 
\begin{equation}\label{small_perimeter}
|D\chi|(B^+_{r_i/5}(z_i)) \leq c  (r_i/5)^{d-1}\,.
\end{equation}
In this case,
\begin{align}\label{small_volume}
\begin{split}
\L^d(M_\chi\cap B_{r_i/5}(z_i))&\leq \L^d(M_\chi\cap B_{r_i}(z_i))=\tau_{\nu}(\tau_{\nu}^{-\frac{1}{2d}}|G|^{-1}r_i^{-1/2})r_i^{d}\\
&\leq \tau_{\nu} r_i^d\leq c r_i^d\,.
\end{split}
\end{align} 
by choosing $\tau_\nu\leq c$. In particular, \eqref{small_perimeter} and \eqref{small_volume} imply that the condition \eqref{small_volume_and_perimeter_of_minority} of Proposition \ref{elastic_lower_bound_on_half_balls} is satisfied in this case (for the spherical cap $B^+_{r_i/5}(z_i)$). Hence, by using again the second case in \eqref{eq: cases_for_R_x} and  \eqref{inverse_volume_bounds}, we infer that
\begin{align}\label{eq: 3_pf_of_energy_lower_bound}
\E\big((u,\chi);B^+_{r_i/5}(z_i)\big)&\geq \int_{B^+_{r_i/5}(z_i)}|\nabla u- \chi G|^2\,\mathrm{d}x \geq c r_i^{-d}|G|^2\L^d(M_\chi\cap B^+_{\gamma r_i/5}(z_i))^2\nonumber\\[2pt]
&\geq c\tau_\nu |G|^{2-\frac{2d}{2d-1}}\L^d(M_\chi\cap B_{r_i}(z_i))^{-\frac{2d}{2d-1}}\L^d(M_\chi\cap B^+_{\gamma r_i/5}(z_i))^2\nonumber\\[2pt]
&\geq c\tau_\nu (c_{1,d,\nu})^{\frac{2d}{2d-1}} |G|^{\frac{2d-2}{2d-1}}\big[\L^d(M_\chi\cap B^+_{\gamma r_i/5}(z_i))\big]^{\frac{2d-2}{2d-1}}\,.
\end{align}
Finally, choosing $\tau_{\nu}=c=c(d,\nu)$, by \eqref{eq: 1_pf_of_energy_lower_bound} with \eqref{eq: c_1_2_d_nu_explicit}, \eqref{eq: 2_pf_of_energy_lower_bound} and \eqref{eq: 3_pf_of_energy_lower_bound} imply altogether that in all possible cases \eqref{inverse_energy_bounds} holds true, with a constant $c_{2,d,\nu}$ which can be made explicit, namely
\begin{align*}
c_{2,d,\nu}&\sim\min\big\{(c_{1,d,\nu})^{\frac{d-1}{d}},c\tau_{\nu}^{-\frac{d-1}{d}}, c\tau_\nu (c_{1,d,\nu})^{\frac{2d}{2d-1}}\big\}\\[3pt]
&\sim \min\big\{\gamma^{d-1},c^{1/d}, c^2\gamma^\frac{2d^2}{2d-1}\big\}=c^2\gamma^\frac{2d^2}{2d-1}\,,
\end{align*} 
for which we recall Remark \ref{rem:power_law_lower} and that $\gamma\sim\nu_1$. This completes the proof of \eqref{inverse_energy_bounds} with the constant of \eqref{eq: c_2_d_nu_explicit}, and thus of the lower bound. \qed\\[-8pt] 
\end{step}

\begin{remark}\label{rem_lower_power_law}
\normalfont 
For every $\nu\in \S^{d-1}$, $G=a\otimes \nu\in \mathcal{R}_1(d)$ and $\chi\in BV(\R^d_+;\{0,1\})$ let us set (recalling \eqref{eq: rescaled_energy}),
\[\mathcal{E}^{\nu}(\chi):=\underset{u\in H^1(\R^d_+;\R^d)}{\inf}\bigg(\int_{\R^d_+}|\nabla u-\chi G|^2\bigg)+|D\chi|(\R^d_+)\,, \ \text{and } E^{\nu}(\mu):=\underset{\int_{\R^d_+} \chi=\mu}{\inf}\mathcal{E}^{\nu}(\chi)\,.\]
In the case $\mu\geq |G|^{-2d}$, 
we obtained a lower bound of the form 
\begin{equation}\label{eq:scaling_law_forE_nu}
E^{\nu}(\mu)\geq C_1(d,\nu)|G|^\frac{2d-2}{2d-1}\mu^{\frac{2d-2}{2d-1}}\,.
\end{equation}
Without even tracking the explicit dependence of the constants appearing in the proofs on $\nu$, it is 
clear that such a constant should have the property that
\begin{equation}\label{eq: C_1_goes_to_0_as_nu_to_0}
C_1(d,\nu)\to 0 \ \ \mathrm{as} \ \ \nu\to \pm e_d\,.
\end{equation}
Indeed, let $(\nu_k)_{k\in \N}\subset \S^{d-1}\setminus\{\pm e_d\}$ be such that $\nu_k\to \pm e_d$ as $k\to \infty$, and up to passing to a subsequence we may assume that 
\[\lim_{k\to \infty} C_1(d,\nu_k) \text{ exists in } [0,+\infty]\,.\] 
Let also $(\chi_j)_{j\in \N}$ be an infimising sequence for $E^{\pm e_d}(\mu)$. Then, using also \eqref{eq: rescaled_specific_direction_energy_scaling} (whose proof will be given in Section \ref{sec:reflection}), we have

\begin{align*}
\mu^{\frac{3d-3}{3d-1}}&\gtrsim E^{\pm e_d}(\mu)=\lim_{j\to \infty}\E^{\pm e_d}(\chi_j)=\lim_{j\to\infty}\lim_{k\to\infty} \mathcal{E}^{\nu_k}(\chi_j)\\[3pt]
&\geq \lim_{j\to\infty}\lim_{k\to\infty} E^{\nu_k}(\mu) = \lim_{k\to\infty} E^{\nu_k}(\mu)\geq \big(\lim_{k\to\infty} C_1(d,\nu_k)\big)\mu^{\frac{2d-2}{2d-1}}\,,
\end{align*} 
\textit{i.e.},
\begin{equation*}
\lim_{k\to\infty} C_1(d,\nu_k)\lesssim \mu^{\frac{-(d-1)}{(2d-1)(3d-1)}}\,,
\end{equation*} 
and by the arbitrariness of $\mu\geq |G|^{-2d}$ and the (sub)-sequence $(\nu_k)_{k\in \N}$ we infer \eqref{eq: C_1_goes_to_0_as_nu_to_0}. 

In particular, for $\nu=\nu_1e_1+\nu_d e_d$ with $\nu_1>0$, (without restriction), collecting \eqref{eq: c_1_2_d_nu_explicit}, \eqref{eq: c_2_d_nu_explicit}, \eqref{eq_lower_bound_pf}, \eqref{eq:explicit_c_d_nu} and \eqref{eq: choice_of_theta_gamma}, and reasoning as in Remark \ref{rem:power_law_lower} we obtain an explicit dependence 
\begin{equation}\label{eq:explicit_lower_bound_constant}
C_1(d,\nu)\sim c^2\gamma^\frac{2d^2}{2d-1}(\gamma^d)^{\frac{2(d-1)}{2d-1}}=c^2\gamma^{2d}\sim \begin{cases}
	\dist^6(\nu,\{\pm e_d\})\qquad  \mathrm{ if }\ d=2\,,\\[3pt]
	\dist^{2d+6}(\nu,\{\pm e_d\})\  \ \mathrm{ if }\ d\geq 3\,.
\end{cases}
\end{equation}
Although this dependence is immediate by careful inspection of the proof, we are not aware of the optimal constant for the lower bound inequality in \eqref{eq:scaling_law_forE_nu} in terms of scaling with respect to $\nu_1$. Actually, as we will see in the next subsection, the constant in the upper bound construction also degenerates to $0$ as $\nu\to\pm e_d$, but with a different (much slower) power law. 
\end{remark}
\subsection{The upper bound}\label{sec:4.2}
The proof of the upper bound in the case of small inclusions, \textit{i.e.,} when $\mu\leq |G|^{-2d}$, relies on essentially the same construction for the upper bound in \cite[Theorem 3.6]{knupfer2011minimal}, here in the case of full gradients and for inclusions supported on $\R^d_+$. For inclusions with volume $\mu\geq |G|^{-2d}$, the construction is more delicate, to take into account also the role of the direction of the rank-1 connection.

\subsubsection {The case of small inclusions}\label{upb_small_incl}
For an inclusion of volume $\mu\leq |G|^{-2d}$, take $\chi_R:=\chi_{B^+_R}$ (centered at 0), with $R=(\frac{2}{\omega_d})^{1/d}\mu^{1/d}$ and consider the displacement field $u\in C^\infty_c(\R^d_+;\R^d)$ defined by $u_R(x):=\zeta_R(x)Gx$\,, where $\zeta_R\in C^\infty_c(\R^d;\R_+)$ is a smooth cut-off function such that
\[0\leq \zeta_R\leq 1\,, \quad \zeta_R|_{B_R}\equiv 1\,, \quad \zeta_R|_{\R^d\setminus \overline{B}_{2R}}\equiv 0\,, \quad \|\nabla \zeta_R\|_{L^\infty(\R^d)}\lesssim 1/R\,.\]
Noting that $\nabla u_R(x)=\zeta_R(x)G+Gx\otimes \nabla\zeta_R(x)$, that $\zeta_R\equiv \chi_R$ in $B_R^+\cup (\R^d_+\setminus \overline{B}^+_{2R})$ and that $\nabla \zeta_R\equiv 0$ on the same set, we calculate
\begin{align}\label{1_energy_upper_bound}
\begin{split}
E(\mu)&\leq  \int_{\R^d_+}|\nabla u_R-\chi_R G|^2\,\mathrm{d}x+|D\chi_R|({\R^d_+})\\[4pt]
&=\int_{B^+_{2R}\setminus B^+_R}\big|(\zeta_R-\chi_R)G+Gx\otimes \nabla \zeta_R(x)\big|^2\,\mathrm{d}x+|D\chi_R|(\R^d_+)\\[4pt]
&\lesssim \big(\|\zeta_R\|^2_{L^\infty}+\|\chi_R\|^2_{L^\infty}+R^2\|\nabla\zeta_R\|^2_{L^\infty}\big) |G|^2\L^d(B^+_{2R}\setminus B^+_{R})+|D\chi_R|(\R^d_+)\\[4pt]
&\lesssim |G|^2 R^d+R^{d-1}\sim (1+|G|^2\mu^{1/d})\mu^{\frac{d-1}{d}}\lesssim \mu^{\frac{d-1}{d}}\,,
\end{split}
\end{align}
where in the last step we used the fact that $\mu\leq |G|^{-2d}$. Hence, \eqref{1_energy_upper_bound} gives the energy upper bound in the first case of \eqref{eq: rescaled_generic_energy_scaling} and \eqref{eq: rescaled_specific_direction_energy_scaling}.

\subsubsection{The case of large inclusions (two-dimensional construction)}	
In this subsection, and for an inclusion of volume $\mu\geq |G|^{-2d}$, we provide a piecewise affine construction whose energy contribution matches the lower scaling bounds found in the previous subsection in terms of $\mu$, despite the different scaling of the multiplicative prefactor in terms of $\nu$.
Our construction takes inspiration from the one in \cite[Section 4.2]{Tribuzio-Rueland_1} and can be seen as a piecewise affine version of that in \cite{knupfer2011minimal}.  

Recalling the notation in \eqref{eq: rescaled_energy}--\eqref{eq: rescaled_minimal_energy} and for presentation simplicity, we first give the argument in two dimensions and then generalize it for any $d\ge2$.

Let us denote again $\nu:=(\nu_1,\nu_2)\neq \pm e_2$ the rank-1 direction and $\nu^\bot:=(-\nu_2,\nu_1)$ the corresponding orthogonal vector. Without loss of generality, (unlike our working assumption in Step 2 of the proof of Proposition \ref{elastic_lower_bound_on_half_balls}), we will assume here without restriction that $\nu_1<0\le \nu_2$.
 
\medskip

\textbf{Definition of the inclusion domain:}
Let $1<H\le L$ be two parameters which respectively represent the thickness and the length of the thin inclusion domain.
We define the set
\begin{equation}\label{eq:domain}
M:=\Big\{z_1\nu^\bot+z_2 \nu : |z_1|< L, |z_2|< H-\frac{H}{L}|z_1|\Big\}\cap\R^2_+\,.
\end{equation}
In particular $M$ is obtained as the portion in the upper half-plane of a rhombus centered at $0$, with diagonals in the $\nu$ and $\nu^\bot$ direction, of length $2H$ and $2L$ respectively. 

\begin{figure}[htb]
\begin{center}
\includegraphics{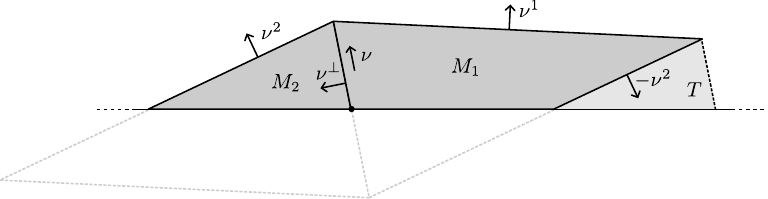}
\caption{A representation of the inclusion domain $M$ and of the outer normal vectors to $M_1$ and $M_2$.}
\label{fig:2d-up1}
\end{center}
\end{figure}

For later use, we also denote the subdomains 
\[M_1:=M\cap\{x\in\R^2_{+} : x\cdot\nu^\bot<0\} \ \text{and } M_2:=M\setminus\overline{M_1}\,.\]
The outer normal vector to $\partial M_1\cap\R^2_+$ and $\partial M_2\cap\R^2_+$ takes values respectively in $\{\nu^1, -\nu^2, \nu^\bot\}$ and $\{\nu^2, -\nu^\bot\}$, where
$$
\nu^1:=c\Big(\nu-\frac{H}{L}\nu^\bot\Big),
\quad
\nu^{2}:=c\Big(\nu+\frac{H}{L}\nu^\bot\Big)\,,
$$
with $c:=\big(1+\frac{H^2}{L^2}\big)^{-\frac{1}{2}}$, see Figure \ref{fig:2d-up1}.

\medskip

\textbf{Choice of the deformation gradients:}
Let $v_0:M\to\R^2$ be a piecewise affine function with
$$
\nabla v_0(x)=\begin{cases}
G_1 & \text{if } x\in M_1\,, \\
G_2 & \text{if } x\in M_2\,.
\end{cases}
$$
We choose $G_1$ and $G_2$ such that $v_0$ is continuous and attains zero value on the upper part of the boundary of $M$, \textit{i.e.}, 
\[v_0\equiv0 \text{ on } \partial M\cap\{x\in\R^2_+ :  x\cdot \nu\ge 0 \}\,.
\]
By Hadamard's jump condition this amounts to finding $a^1,a^2,a^3\in \R^2\setminus\{0\}$ such that
$$
G_1=a^1\otimes \nu^1,
\quad
G_2=a^2\otimes \nu^2,
\quad
G_2-G_1=a^3\otimes\nu^\bot\,.
$$
The latter is solved by $a^2=a^1$ and $a^3=2c\frac{H}{L}a^1$.
Hence, recalling that $G=a\otimes \nu$, we define the matrices
\begin{equation*}
G_1:=\frac{1}{c}a\otimes \nu^1\,, \quad
G_2:=\frac{1}{c}a\otimes \nu^2\,.
\end{equation*}
The choice $a^1=\frac{1}{c}a$ minimises (in terms of scaling in $H$ and $L$) the quantities $|G_1-G|$ and $|G_2-G|$ and keeps the computations simple.
With these choices, the deformation $v_0$ reads
\begin{equation}\label{eq: v:0}
v_0(x):=\begin{cases}
G_1x-Ha & \text{if } x\in M_1\,, \\
G_2x-Ha & \text{if } x\in M_2\,.	 
\end{cases}
\end{equation}

\medskip

\textbf{Extension outside the domain:}
We now extend $v_0$ outside $M$.
Since $v_0\equiv 0$ on $\partial M\cap\{x\in\R^2_+ : x\cdot \nu\ge0\}$, we are left with extending $v_0$ across the boundary \[\partial M\cap\{x\in\R^2_+ : x\cdot \nu\le0\}\,.\]
We then look for $\tilde G\in\R^{2\times 2}$ that is rank-1 connected to $G_1$ in the direction $\nu^2$ and to $0$ in some direction $\tilde \nu\neq \pm\nu$, to ensure that this extension has compact support.
Thus, we require 
\[\tilde G=\tilde a \otimes\tilde \nu \ \text{ and } G_1-\tilde G=a^4\otimes \nu^2\,,\] for some $\tilde a, a_4 \in\R^2\setminus\{0\}$ and $\tilde \nu \in\S^1\setminus\{\pm\nu\}$.
This is solved by $\tilde a=-\frac{2H}{L}a$, $a^4=\frac{1}{c}a$ and $\tilde \nu=\nu^\bot$, \textit{i.e.},
\begin{equation*}
\tilde G := -\frac{2H}{L}a\otimes\nu^\bot\,.
\end{equation*}
Via \eqref{eq: v:0}, we define $v:\R^2_+\to\R^2$ as
\begin{equation}\label{eq:v_deformation}
v(x):=\begin{cases}
v_0(x) & \text{if } x\in M\,, \\[2pt]
\tilde Gx- 2Ha & \text{if } x\in T\,, \\[2pt]
0 & \text{otherwise}\,,
\end{cases}
\end{equation}
where $T$ is the set
$$
T:=\Big\{-z_1\nu^\bot+z_2 \nu : 0<z_1<L, z_2< -H+\frac{H}{L}z_1\Big\}\cap\R^2_+\,.
$$
Eventually, via \eqref{eq:v_deformation}, we define the deformation  $u\in W^{1,\infty}(\R^2_+;\R^2)$ as 
\begin{equation}\label{eq: v_def}
u(x):=\zeta(x)v(x)\,,
\end{equation}
where $\zeta\in C^\infty_c(B_{4L};[0,1])$ is a cut-off function such that 
\begin{equation}\label{eq: smooth_cut_off}
\zeta\equiv1 \text{ in } B_{2L},\ \ \zeta\equiv0  \text{ in } \R^2\setminus\overline{B}_{4L}\, \ \ \text{and } \ \|\nabla\zeta\|_{L^\infty}\le \frac{1}{L}\,.
\end{equation}

Notice that this additional cut-off is needed only for the cases $\nu_2\ll 1$, in particular it does not affect the deformation (\textit{i.e.}, $u\equiv v$) when $\nu_2>\frac{1}{2}$.

\medskip

\textbf{Energy contribution:}
Now we compute $\mathcal{E}(u,\chi)$ for $u$ as in \eqref{eq: v_def} and $\chi:=\chi_M$ under the constraint $\|\chi\|_{L^1(\R^d_+)}=\mu\geq |G|^{-2d}$.
Preliminarily, we notice that
\begin{equation*}\label{eq:volume1}
\mu = \L^2(M) = HL\,,
\end{equation*}
and that
\begin{equation}\label{eq:volume2}
\L^2(T\cap B_{4L}) 
\leq  
\theta(\nu)L^2,	\ \text{with } \theta(\nu):=\min\Big\{\frac{1}{2\nu_2},4\Big\}|\nu_1|
\,.
\end{equation}
Indeed, the vertices of $T$ are 
\[P_1=(L\nu_2,-L\nu_1),\ P_2=\Big(\frac{HL}{H\nu_2-L\nu_1}
,0\Big), \ \text{ and } P_3=(\frac{L}{\nu_2},0)\,,\]
which yields that 
\[\L^2(T)= \frac{L^3\nu_1^2}{2\nu_2(H\nu_2+L|\nu_1|)}\le\frac{|\nu_1|}{2\nu_2}L^2\,.\]
Noticing additionally that $B_{4L}\cap T\subset B_{4L}\cap([0,4L]\times[0,L|\nu_1|])$ and that 
\[\L^2(T\cap B_{4L})\le\min\big\{\L^2(T),\L^2\big([0,4L]\times[0,L|\nu_1|]\big)\big\}\,,\] \eqref{eq:volume2} follows.
We compute the surface and the elastic energy separately.
For the former, it holds
\begin{equation}\label{eq:surf-en}
|D\chi|(\R^2_+) = \mathcal{H}^1(\partial M\cap \R^2_+) = 2(H^2+L^2)^\frac{1}{2} \le 3L\,.
\end{equation}
For the latter, recalling \eqref{eq: v:0}--\eqref{eq: v_def}, noticing that $\|v\|_{L^\infty}\le 4 H|a|$, which follows from the fact that 
\begin{align*}
G_1x=(x\cdot \nu-\tfrac{H}{L}x\cdot\nu^\bot)a\,, \ G_2x=(x\cdot \nu+\tfrac{H}{L}x\cdot\nu^\bot)a\,, \ \tilde Gx=-\frac{2H}{L}(x\cdot \nu^\bot)a\,,
\end{align*}
and that $|x\cdot \nu|\leq H, |x\cdot \nu^\bot|\leq L$ for all $x\in M\cup T$,  by \eqref{eq:volume2} and that $M\subset B_{2L}$, we have
\begin{equation}\label{eq:el-en}
\begin{split}
\int_{\R^2_+} |\nabla u-\chi G|^2 &= \int_M |\nabla v-G|^2 + \int_{T\cap B_{2L}} |\nabla v|^2 + \int_{(T\cap B_{4L})\setminus B_{2L}}|\nabla u|^2 \\[2pt]
&\le \L^2(M_1)|G_1-G|^2+\L^2(M_2)|G_2-G|^2 + \L^2(T\cap B_{\GGG 2L})|\tilde G|^2\\[2pt]
&\quad  + \L^2\big((T\cap B_{4L})\setminus B_{2L}\big)\|\nabla\zeta\|_{L^\infty(\R^2_+)}^2\|v\|_{L^\infty(\R^2_+)}^2\\[2pt]
&\le \Big(\frac{H^3}{L}+32|\nu_1|H^2\Big)|a|^2\,,
\end{split}
\end{equation}
where we have additionally used that $G_2-G=-(G_1-G)=\frac{H}{L}a\otimes \nu^\bot$.
Gathering \eqref{eq:surf-en} and \eqref{eq:el-en}, we obtain
$$
E(\mu) \le C\Big(|a|^2\frac{H^3}{L}+|a|^2|\nu_1|H^2+L\Big),
$$
with $\mu=HL$, for some universal constant $C>1$.
We now optimize the parameters $H$ and $L$ in terms of $\mu$.
We distinguish two cases:
\begin{itemize}
\item[(i)] If $\frac{H^3}{L}>|\nu_1|H^2\iff H>|\nu_1|L$, then 
\[E(\mu)\le C\Big(2|a|^2\frac{H^3}{L}+L\Big)=C\bigg(2|a|^2\frac{\mu^3}{L^4}+L\bigg)\,.\]
The right-hand side is optimized as $L=2^{3/5}|a|^{2/5}\mu^{3/5}$, from which we obtain 
\[E(\mu) \lesssim |a|^{2/5}\mu^{3/5}\,.\]
\item[(ii)] If $\frac{H^3}{L}\le |\nu_1|H^2\iff H\leq |\nu_1|L$,  then \[E(\mu)\le C\big(2|a|^2|\nu_1|H^2+L\big)=C\bigg(2|a|^2|\nu_1|\frac{\mu^2}{L^2}+L\bigg)\,.\]
The right-hand side is now optimized as $L=2^{2/3}|a|^{2/3}|\nu_1|^{1/3}\mu^{2/3}\,,$ from where we obtain 

\[E(\mu) \lesssim |a|^{2/3}|\nu_1|^{1/3}\mu^{2/3}\,.\]
\end{itemize}
Eventually, with these choices of $H$ and $L$, we collect all the previous computations in the following statement.

\begin{proposition}\label{prop:ub2D}
Let $\nu\in\S^1\setminus\{\pm e_2\}$ and $\mu>|G|^{-4}$. 
There exists a universal constant $C>1$ such that
\begin{equation}\label{eq: upper_bound_2dim}
E(\mu) \le C \begin{cases}
|G|^{2/5} \mu^{3/5}\,, \ \quad \qquad \text{ if } |\nu_1|\leq (\frac{|G|^{-4}}{\mu})^{1/5}\\[5pt]
|\nu_1|^{1/3}|G|^{2/3}\mu^{2/3} \,, \ \ \text{ if } |\nu_1|\geq (\frac{|G|^{-4}}{\mu})^{1/5}
\end{cases}\,.
\end{equation}
\end{proposition}

\subsubsection{Large inclusions in general dimension}

We now generalize the construction to every dimension.
Let $\{\nu^\bot_1,\dots,\nu^\bot_{d-1}\}$ be an orthonormal base of $\nu^\perp$ with $\nu^\bot_1\in\mathrm{span}\{\nu,e_d\}$.
Again, we consider $1<H\le L$ and define
\begin{multline}\label{eq:def_M_higher_dim}
M := \Big\{z_1\nu^\bot_1+\dots+z_{d-1}\nu^\bot_{d-1}+z_d \nu : \\
|z_1|+\dots+|z_{d-1}|<L, \, |z_d|<H-\frac{H}{L}(|z_1|+\dots+|z_{d-1}|)\Big\}\cap\R^d_+.
\end{multline}
We consider a scalar function $w_0:M\to\R$, defined as
\begin{equation}\label{eq:w}
w_0(x) = x\cdot \nu + \frac{H}{L}\big(|x\cdot \nu^\bot_1|+\dots+|x\cdot \nu^\bot_{d-1}|\big)-H\,.
\end{equation}
This is a continuous, piecewise affine function attaining $0$ on $\partial M\cap \{x\in\R^d_+ : x\cdot \nu\ge0\}$ with gradient
\begin{equation}\label{eq:grad-w}
\nabla w_0(x) = \nu+\frac{H}{L}\big(\sgn(x\cdot\nu^\bot_1)\nu^\bot_1+\dots+\sgn(x\cdot\nu^\bot_{d-1})\nu^\bot_{d-1}\big)\,.
\end{equation}
We extend $w_0$ to a function $w$ outside $M$. On $\{x\in\R^d_+ : x\cdot \nu\ge0\}\setminus M$ we can extend it by $0$.
On $\{x\in\R^d_+ : x\cdot \nu\le0\}\setminus M$ instead, we extend it constantly in direction $\nu$, \textit{i.e.}, for every $x_0\in\partial M\cap\{x\in\R^d_+ : x\cdot \nu\le0\}$ and $t>0$, we define
\begin{align*}
w(x_0-t\nu) := w_0(x_0) &= x_0\cdot \nu+\frac{H}{L}(|x_0\cdot\nu^\bot_1|+\dots+|x_0\cdot\nu^\bot_{d-1}|)-H \\
	&= \frac{2H}{L}(|x_0\cdot\nu^\bot_1|+\dots+|x_0\cdot\nu^\bot_{d-1}|)-2H\,,
\end{align*}
setting $w= 0$ in the remaining portion of $\R^d_+$. 
We have then defined $w:\R^d_+\to\R$ as
\begin{equation}\label{eq:def_w}
w(x):=\begin{cases}
x\cdot \nu + \frac{H}{L}(|x\cdot \nu^\bot_1|+\dots+|x\cdot \nu^\bot_{d-1}|)-H & \text{for } x\in M\,, \\[4pt]
\frac{2H}{L}(|x\cdot \nu^\bot_1|+\dots+|x\cdot \nu^\bot_{d-1}|)-2H & \text{for } x\in T\,, \\[4pt]
	0 & \text{otherwise}\,,
\end{cases}
\end{equation}
where $T$ is the set
$$
T = \{ x\in\R^d_+ : x\cdot \nu\le 0\,,\ x-x\cdot \nu\in M\},
$$
which is the set points that project on $M$ in direction $\nu$.
Eventually, we define the deformation $u:\R^d_+\to\R^d$ as
\begin{equation}\label{eq:def_u_higher_dim}
u(x) := \zeta(x) w(x)a\,,
\end{equation}
with $\zeta$ a smooth cut-off function as that defined in \eqref{eq: smooth_cut_off}.

\medskip

\textbf{Energy contribution:}
We compute $\mathcal{E}(u,\chi)$ for $u$ defined as in \eqref{eq:def_u_higher_dim}, and $\chi:=\chi_M$ as in \eqref{eq:def_M_higher_dim}, with $\|\chi\|_{L^1(\R^d_+)}=\mu$.

We begin by noticing that $\mu=\L^d(M)\sim HL^{d-1}$ and that, similarly to \eqref{eq:volume2}, 
\begin{equation}\label{eq:genDvolume}
\L^d(T)	\lesssim \theta(d,\nu)
L^d,
\ \ \text{where } \theta(d,\nu):=\min\Big\{\frac{1}{|\nu\cdot e_d|},1\Big\}|\nu^\bot_1\cdot e_d|\,.
\end{equation}
The surface energy term is given by
\begin{equation}\label{eq:genDsurf-en}
|D\chi|(\R^d_+)=\mathcal{H}^{d-1}(\partial M\cap\R^d_+)\sim L^{d-1}.
\end{equation}
By \eqref{eq:grad-w} and \eqref{eq:genDvolume} and by noticing that the definition \eqref{eq:def_w} implies $\|w\|_{L^\infty(\R^d_+)}\lesssim H$, for the elastic energy we have
\begin{equation}\label{eq:genDel-en}
\begin{split}
\int_{\R^d_+}|\nabla u-\chi G|^2\dx &= \int_{M}|a\otimes(\nabla w_0(x)-\nu)|^2\dx + \int_{T\cap B_{2L}}|a\otimes\nabla w(x)|^2\dx \\[3pt]		&\qquad + \int_{(T\cap B_{4L})\setminus B_{2L}}|\zeta(x)a\otimes\nabla w(x)+w(x)a\otimes\nabla\zeta(x)|^2\dx\\[3pt]
&\lesssim \big(H^3L^{d-3}+\theta(d,\nu)H^2L^{d-2}\big) |a|^2\,.
	\end{split}
\end{equation}
Thus, combining \eqref{eq:genDsurf-en} and \eqref{eq:genDel-en}, the total energy contribution reads
$$\mathcal{E}(u,\chi) \lesssim L^{d-2}\Big(|a|^2\frac{H^3}{L}+|a|^2\theta(d,\nu)H^2+L\Big)\,.
$$
Again, we split into two cases:
\begin{itemize}
\item[(i)] If $\frac{H^3}{L}>\theta(d,\nu)H^2\iff H>\theta(d,\nu)L$, then
\[E(\mu) \lesssim L^{d-2}\Big(|a|^2\frac{H^3}{L}+L\Big)\sim |a|^2\mu^3L^{-2d}+L^{d-1}\,.\]
The right-hand side is optimized as $L\sim |a|^{2/3d-1}\mu^{3/3d-1}$, from which, we obtain
\[E(\mu) \lesssim |a|^\frac{2d-2}{3d-1}\mu^\frac{3d-3}{3d-1}\,.\]
\smallskip 
\item[(ii)] If $\frac{H^3}{L}\le \theta(d,\nu)H^2\iff H\le \theta(d,\nu)L$, then $$E(\mu)\lesssim L^{d-2}\big(|a|^2\theta(d,\nu)H^2+L\big)\sim |a|^2\theta(d,\nu)\mu^2L^{-d}+L^{d-1}\,.$$
The right-hand side is now optimized as $L\sim |a|^{2/2d-1}\theta(d,\nu)^{1/2d-1}\mu^{2/2d-1}
$, from which we obtain
$$
E(\mu)\lesssim |a|^\frac{2d-2}{2d-1}\theta(d,\nu)^\frac{d-1}{2d-1}\mu^\frac{2d-2}{2d-1}.
$$
\end{itemize}
We recollect the previous computations and the result from Proposition \ref{prop:ub2D} in the following statement. 

\begin{proposition}\label{prop:ub-genD}
Let $d\geq 2$, $\nu\in\S^{d-1}\setminus\{\pm e_d\}$ and $\mu>|G|^{-2d}$, and let

\begin{equation}\label{eq:def_theta_d_nu}
\theta(d,\nu):=\min\Big\{\frac{1}{|\nu\cdot e_d|},1\Big\}|\nu^\bot_1\cdot e_d|\,.
\end{equation}
Then, there exists a dimensional constant $C_d>1$ such that

\begin{equation}\label{eq: upper_bound_higher_dim}
E(\mu) \le C_d\begin{cases} |G|^\frac{2d-2}{3d-1}\mu^\frac{3d-3}{3d-1} \qquad \qquad \quad \text{ if } \theta(d,\nu)\leq (\frac{|G|^{-2d}}{\mu})^{\frac{1}{3d-1}}\\[5pt]
\theta(d,\nu)^\frac{d-1}{2d-1}|G|^\frac{2d-2}{2d-1}\mu^\frac{2d-2}{2d-1} \ \text{ if } \theta(d,\nu)\geq (\frac{|G|^{-2d}}{\mu})^{\frac{1}{3d-1}}\,.
\end{cases}
\end{equation}
\end{proposition}

Note that even in the first case of \eqref{eq: upper_bound_higher_dim}, in which 
$$|G|^\frac{2d-2}{3d-1}\mu^\frac{3d-3}{3d-1}\geq \theta(d,\nu)^\frac{d-1}{2d-1}|G|^\frac{2d-2}{2d-1}\mu^\frac{2d-2}{2d-1}\,,$$
using the fact that we are at the regime $\mu>|G|^{-2d}$, we can estimate
\begin{equation}\label{eq: energy_barrier_close_to_e_d}
|G|^\frac{2d-2}{3d-1}\mu^\frac{3d-3}{3d-1}\leq |G|^\frac{2d-2}{2d-1}\mu^\frac{2d-2}{2d-1}\,.
\end{equation} 
 
The proof of the upper bound in \eqref{eq: rescaled_generic_energy_scaling} follows then immediately from \eqref{eq: upper_bound_higher_dim} and \eqref{eq: energy_barrier_close_to_e_d}.

	
\begin{remark}[Dichotomy of scalings]
\normalfont	The scaling from Proposition \ref{prop:ub-genD} shows a dichotomy  in a \emph{quantitative} way, depending on the interaction between the anisotropy direction $\nu$ and the normal $e_d$ to the constraining hyperplane.
Indeed, since \eqref{eq:def_theta_d_nu} gives
\begin{equation}\label{theta_nu_degeneracy}
\lim_{\nu\to \pm e_d}\theta(d,\nu)=0\,,
\end{equation}
by \eqref{eq: upper_bound_higher_dim} we infer the following:  Loosely speaking, when $\dist(\nu,\{\pm e_d\})$ is sufficiently small (in terms of the ratio $|G|^{-2d}/\mu$), then the upper bound for $E(\mu)$ coincides with the nucleation barrier of the three-well case, studied in \cite{Tribuzio-Rueland_1}.
Otherwise, the scaling of the upper bound (in terms of $|G|$ and $\mu$ )is the one of the unconstrained problem studied in \cite{knupfer2011minimal}, \textit{i.e.,} in that case $E(\mu)\sim_{\nu,G} \mu^\frac{2d-2}{2d-1}$. This can also be viewed in the following way:
\begin{itemize}
\item[(i)] If $\nu\in\S^{d-1}\setminus\{\pm e_d\}$ is fixed then, up to multiplicative constants depending on $d$ and $G$, there holds
$$E_\infty(\nu):=\lim_{\mu\to+\infty}\mu^{-\frac{2d-2}{2d-1}}E(\mu) \lesssim \dist(\nu,\{\pm e_d\})^\frac{d-1}{2d-1}\,.$$ 
Taking in the latter the limit as $\nu\to \pm e_d$, and using \eqref{theta_nu_degeneracy}, we get
$$
\lim_{\nu\to \pm e_d} E_\infty(\nu)=0\,.$$
We can interpret this by saying that the scaling of the two-well problem degenerates as $\nu\to \pm e_d$.\\[2pt]
\item[(ii)] If $\mu>|G|^{-2d}$ is fixed, then the scaling of the upper bound in \eqref{eq: upper_bound_higher_dim} is (up to multiplicative prefactors) 
$\mu^\frac{3d-3}{3d-1}$, 
as is the scaling for the three-well problem.
\end{itemize}
\end{remark}

\begin{remark}(Dependence of the energy barrier on $\nu$)
\normalfont As already discussed in Remark \ref{rem_lower_power_law} of the previous subsection, for $\nu\neq \pm e_d$, we see that the lower and upper bounds for the energy barrier $E^{(\nu)}(\mu)$ obey of course the same scaling law in terms of $\mu$, up to multiplicative prefactors that have a different power-law behavior in terms of the rank-1 direction $\nu$. In particular, by \eqref{eq:scaling_law_forE_nu}, \eqref{eq:explicit_lower_bound_constant}, \eqref{eq:def_theta_d_nu} and \eqref{eq: upper_bound_higher_dim}, the method of proof of the lower bound yields a constant $C_1(d,\nu)$ that tends to 0 much faster than the corresponding constant obtained in the upper bound construction. We are tempted to conjecture that the power-law behavior of the latter is the optimal one. 
\end{remark}
 
\section{Proof of Theorem \ref{main_theorem_rescaled}$(ii)$}\label{sec:reflection}
For the proof of the second part of Theorem \ref{main_theorem_rescaled}, we see that by means of a simple reflection argument the problem transforms to a three-well compatible problem in $\R^d$, as in \cite[Theorem 2, Corollary 1.1]{Tribuzio-Rueland_1}. To be more precise, let now $G=a\otimes e_d\in \mathcal{R}_1(d)$ and $(u,\chi)\in \mathcal{A}(\mu)$ (recall \eqref{eq: rescaled_admissible_pairs}). 

We define a pair $(\tilde u,\tilde \chi)\in H^{1}(\R^d;\R^d)\times BV(\R^d;\{-1,0,1\})$ as follows: For every point $x\in \R^d$ and setting $B=\mathrm{diag}(1,\dots,1,-1)\in \R^{d\times d}$, let
\begin{equation}\label{eq: reflected_u}
\tilde u(x):=\begin{cases}
u(x) \quad \ \ \ \text{if } x\in \R_+^d\\
u(Bx) \quad \text{if } x\in \R_-^d\,,
\end{cases}	\implies  \nabla \tilde u(x):=\begin{cases}
\nabla u(x) \quad \ \ \ \ \ \text{if } x\in \R_+^d\\
\nabla u(Bx)B \quad \text{if } x\in \R_-^d\,,
\end{cases}	
\end{equation}
and
\begin{equation}\label{eq: reflected_chi}
\tilde \chi(x):=\begin{cases}
\chi(x) \qquad \ \  \text{if } x\in \R_+^d\\
-\chi (Bx) \quad \text{if } x\in \R_-^d\,.
\end{cases}
\end{equation}	
Recalling \eqref{eq: jump_set_of_chi}, and setting also $J_{\tilde \chi}:=\mathrm{spt}(D\tilde \chi)$, we have by construction
\begin{equation}\label{eq: spt_of_D_refl_chi}
J_{\tilde \chi}:=J_\chi\cup (BJ_\chi)\cup \big(\partial^*\{\chi=1\}\cap \{x_d=0\}\big)\,.
\end{equation} 
Via this extension, the total energy of the pair $(\tilde u, \tilde \chi)$, \textit{i.e.}, 
\begin{equation}\label{eq:total_energy_of_the_extended_pair}
\tilde{\mathcal{E}}(\tilde u,\tilde \chi):=\int_{\R^d}|\nabla \tilde u-\tilde\chi G|^2\,\mathrm{d}x+|D\tilde \chi|(\R^d)
\end{equation}
is controlled by the energy of $(u,\chi)$ in $\R^d_+$ (recall \eqref{eq: rescaled_energy}), as the following lemma indicates.
\begin{lemma}\label{comparison_of_energies}
Let $(\tilde u,\tilde\chi)\in H^{1}(\R^d;\R^d)\times BV(\R^d;\{-1,0,1\})$ be defined via $(u,\chi)$ as in \eqref{eq: reflected_u} and \eqref{eq: reflected_chi}. Then,
\begin{equation}\label{new_energy_between_2_3_of_old}
2\mathcal{E}(u,\chi)\leq \tilde{\mathcal{E}}(\tilde u,\tilde \chi) \leq 4\mathcal{E}(u,\chi)\,.
\end{equation}
\end{lemma}
\begin{proof}
By the above definitions, and since
\[-GB=-(a\otimes e_d)B=-a\otimes (B^te_d)=-a\otimes (-e_d)=a\otimes e_d=G\,,\]  we immediately have:
\begin{align}\label{eq_elastic_energies_estimates}
\int_{\R^d}|\nabla \tilde u-\tilde \chi G|^2\,\mathrm{d}x&=\int_{\R_+^d}|\nabla u-\chi G|^2\,\mathrm{d}x+\int_{\R_-^d}|\nabla  u(Bx)B+\chi(Bx) G|^2\,\mathrm{d}x \nonumber\\
&\overset{y:=Bx}{=} \int_{\R_+^d}|\nabla u-\chi G|^2\,\mathrm{d}x+\int_{\R_+^d}|\nabla  u(y)B-\chi(y) GB|^2\,\mathrm{d}y \nonumber\\
&=2\int_{\R^d_+} |\nabla u-\chi G|^2\,\mathrm{d}x\,.
\end{align}
Regarding the interfacial energy, it is clear by \eqref{eq: spt_of_D_refl_chi} that 
\begin{equation}\label{eq_interfacial_energy_estimates}
2|D\chi|(\R_+^d)=2\mathcal{H}^{d-1}(J_\chi)\leq |D\tilde \chi|(\R^d)=2|D\chi|(\R_+^d)+2\mathcal{H}^{d-1}\big(\partial^*\{\chi=1\}\cap \{x_d=0\}\big)\,,
\end{equation}
and therefore, to finish the proof of the lemma, it suffices to verify that
\begin{equation}\label{minimal_surface_inequality}
\mathcal{H}^{d-1}\big(\partial^*M_\chi\cap \{x_d=0\}\big)\leq \mathcal{H}^{d-1}\big(\partial^*M_{\chi}\cap \{x_d>0\}\big)=|D\chi|(\R^d_+)\,,
\end{equation}
where $M_\chi:=\{\chi=1\}$. In the classical (smooth) setting, the above inequality is simply the assertion that the area-minimising $(d-1)$-dimensional hypersurface in $\R^d$, spanning a given closed $(d-2)$-dimensional planar surface, is the (bounded) flat surface with that given boundary. We give the proof of \eqref{minimal_surface_inequality} in the context of sets of finite perimeter, based on an inductive argument in the dimension $d\geq 2$.

For $d=2$ the inequality is obvious, since for every indecomposable component $P$ of $M_\chi\cap \{x_2>0\}$ (cf. \cite[Example 4.18]{Ambrosio-Fusco-Pallara:2000}) for which $\mathcal{H}^1(\partial^*P\cap \{x_2=0\})>0$, we have that $\partial^*P\cap \{x_2=0\}$ is up to an $\mathcal{H}^1$-null set an interval with endpoints $(a_P,0), (b_P,0)$ with $a_P<b_P$, while $\partial^*P\cap \{x_2>0\}$ is a $1$-rectifiable curve with the same endpoints.

Assume now for the inductive step that \eqref{minimal_surface_inequality} holds in dimension $d$ and let us prove it in dimension $d+1$. For this purpose, let $M\in \mathcal{B}(\R^{d+1})$ be a set of finite perimeter, with $M\subset \{x\in \R^{d+1}\colon x_{d+1}\geq 0\}$. For $t\in \R$, let us set $$M_t:=M\cap \{x_1=t\}\,,$$
and by the general slicing properties of sets of finite perimeter, we have that $M_t$ is a subset of finite perimeter in $\{x_1=t\}$ for $\L^1$-a.e. $t\in \R$. Hence by the coarea formula (cf. \cite[Equation (18.25)]{maggi2012sets}), the inductive hypothesis (for $M_t\subset \{x_1=t,x_{d+1}>0\}$) and Fubini's theorem, we obtain
\begin{align*}
\H^d(\partial^*M\cap \{x_{d+1}>0\})&=\int_{\partial^*M}\chi_{\{x_{d+1}>0\}}\,\mathrm{d}\H^d\geq \int_{\partial^*M}\chi_{\{x_{d+1}>0\}}\sqrt{1-(\nu_{\partial^*M}\cdot e_1)^2}\,\mathrm{d}\H^d\\[4pt]
&=\int_{\R}\mathrm{d}t\int_{(\partial^*M)_t}\chi_{\{x_{d+1}>0\}}\,\mathrm{d}\H^{d-1}\\[4pt]
&=\int_{\R}\H^{d-1}(\partial^*(M_t)\cap \{x_1=t, x_{d+1}>0\})\,\mathrm{d}t\\[4pt]
&\geq \int_{\R}\H^{d-1}(\partial^*(M_t)\cap \{x_{d+1}=0\})\,\mathrm{d}t\\[4pt]
&= \H^{d}(\partial^*M\cap \{x_{d+1}=0\})\,,
\end{align*}
which completes the induction for proving \eqref{minimal_surface_inequality}. In view of \eqref{eq_elastic_energies_estimates}--\eqref{minimal_surface_inequality}, the proof of \eqref{new_energy_between_2_3_of_old} is complete. 
\end{proof}

The proof of Theorem \ref{main_theorem_rescaled}(ii) follows then immediately by combining Lemma \ref{comparison_of_energies} together with \cite[Corollary 1.1 and Remark 4.4]{Tribuzio-Rueland_1}. 

\section*{Acknowledgements} 
AT's work was funded by the Deutsche Forschungsgemeinschaft (DFG,
German Research Foundation) through SPP 2256, project ID 441068247. KZ was supported by the
Deutsche Forschungsgemeinschaft (DFG, German Research Foundation) under Germany’s Excellence Strategy EXC-2044-390685587, "Mathematics Münster: Dynamics-Geometry-Structure",
and currently by the SFB 1060 and the Hausdorff Center for Mathematics
(HCM) under Germany’s Excellence Strategy -EXC-2047/1-390685813.

The authors would also like to thank the Hausdorff Institute for Mathematics at the University of Bonn which is funded by the DFG under Germany’s Excellence Strategy -EXC-2047/1-390685813, as part of the Trimester Program on Mathematics for Complex Materials. 

\appendix
	
 \section{Proof of properties \eqref{properties_of_tilting}}\label{sec: choice_of_cages}

Without seeking to optimise the dependence of the parameters $\theta,c,\gamma$ in $\nu$ (in terms of prefactors) here, and recalling the working assumption $\nu_2\leq 0<\nu_1$, we show the assertion with 
\begin{equation}\label{eq: choice_of_theta_gamma}
\theta=\gamma=\frac{1}{50}\nu_1\leq \frac{1}{50}\,.
\end{equation}
Recalling the notation \eqref{base_cap_endpoints}, let 
\begin{equation*}
l:=|x_5-x_0|\in [2\sqrt{1-\theta^2},2]\,.
\end{equation*}
Moreover, we denote 
\begin{equation}\label{eq: division_points}
x_1:=x_0+\tfrac{l}{6}e_1\,, \ x_2:=x_0+\tfrac{2l}{6}e_1\,, \ 	x_3:=x_0+\tfrac{4l}{6}e_1\,, \ x_4:=x_0+\tfrac{5l}{6}e_1\,,
\end{equation}
and (recalling \eqref{eq: stupid notation}) for every $\sigma\in [2\gamma,4\gamma]$ 
and $i=1,2,3,4$, 
\begin{equation*}
y_{i,\sigma}:=(x_{i}+\R\nu^\bot)\cap \R^h_\sigma\,, \ \  z_{i,\sigma}:=x_i+\big((y_{i,\sigma}-x_i)\cdot e_1\big)e_1\,,
\end{equation*}
the latter being the projection of the vector $y_{i,\sigma}-x_{i}$ onto the $x$-axis, emanating from $x_i$. Finally, for every $i=1,2,3$,  we denote 
\begin{equation}\label{eq: division_parallelograms}
Q^\sigma_{i}:=\mathrm{conv}\big(\{x_{i}, x_{i+1}, y_{i+1,\sigma},  y_{i,\sigma}\}\big) \ \ \text{and}\ \ Q^\sigma:=\bigcup_{i=1}^3 Q^\sigma_{i}\,.
\end{equation}
\begin{figure}[t]
\begin{center}
\includegraphics{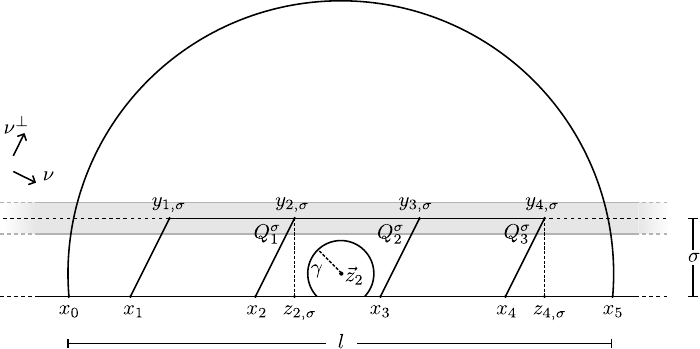}
\caption{Building the \emph{tilted cages}.}
\label{fig:cage2}
\end{center}
\end{figure}
By the simple geometry indicated in Figure \ref{fig:cage2} and recalling that $(\vec{z}_2)_1=0$, in order to ensure that $B_\gamma^+(\vec{z}_2)\subset Q_{2}^\sigma$, it is enough to verify that
\begin{equation}\label{eq: verification_of_inclusion}
\gamma<\min\{|z_{2,\sigma}|, |x_3|\}=|z_{2,\sigma}|=\frac{l}{6}-\big(\frac{-\nu_2}{\nu_1}\big)\sigma=\frac{l}{6}-\frac{\sqrt{1-\nu_1^2}}{\nu_1}\sigma\,.
\end{equation}
For the verification of \eqref{eq: verification_of_inclusion}, recalling the choice in \eqref{eq: choice_of_theta_gamma} and since $0<\nu_1\leq 1$, note that
\begin{align*}
\frac{l}{6}-\frac{\sqrt{1-\nu_1^2}}{\nu_1}\sigma&\geq \frac{1}{3}\sqrt{1-\theta^2}-\frac{4\gamma}{\nu_1}=  \frac{1}{3}\sqrt{1-\frac{\nu_1^2}{ 2500}}-\frac{4}{50}\\
&=\frac{1}{50}\big(\frac{1}{3}\sqrt{ 2499}-4\big)> \frac{1}{50}\geq \gamma\,.
\end{align*}
Similarly, to ensure that $Q^\sigma\subset B_1^+(\vec{z}_2)$, it is enough to check that
\begin{equation}\label{eq: verification_of_second_inclusion}
|z_{4,\sigma}|<\min\Big\{\frac{l}{2},\sqrt{1-\sigma^2}\Big\}=\sqrt{1-\sigma^2}\,.
\end{equation} 
Indeed,
\begin{align*}
|z_{4,\sigma}|&=|x_4|+|z_{4,\sigma}-x_4|=\frac{l}{3}+\frac{(-\nu_2)}{\nu_1}\sigma\leq \frac{l}{3}+\frac{4\gamma}{\nu_1}=  \frac{l}{3}+\frac{4}{50}\,.
\end{align*}
Since
\begin{align*}
\sqrt{1-\sigma^2} \ge \sqrt{1-\frac{16}{2500}\nu_1^2}-\frac{2}{3}>\frac{4}{50}\,,
\end{align*}
the two inequalities in the lines just above imply \eqref{eq: verification_of_second_inclusion}. In particular, the previous elementary arguments show that the points $\{x_i\}_{i=1,2,3,4}$ as in \eqref{eq: division_points} and the corresponding parallelograms $\{Q_i^\sigma\}_{i=1,2,3}$ as in \eqref{eq: division_parallelograms}  satisfy \eqref{properties_of_tilting}(i),(ii). The proof of properties (iii),(iv) therein follow by a Fubini type argument as in Steps 1, 2 of the proof of \cite[Lemma 3.2]{knupfer2011minimal}.

Recalling the notation in \eqref{eq: stupid notation} and \eqref{eq: choice_of_theta_gamma}, for every $s\in[2\gamma,4\gamma]$ let us set 
\begin{equation}\label{eq:length_of_Gamma}
\Gamma^h_s:=\R^h_s\cap B^+_1(\vec{z}_2)\implies 2\ge \L^1(\Gamma^h_s)=2\sqrt{1-(s-z_2)^2}\geq 2\sqrt{1-16\gamma^2}\geq \frac{\sqrt{ 2484}}{25}>1\,.
\end{equation}

Setting also $B^+_{1,\gamma}(\vec{z}_2):=B_1^+(\vec{z}_2)\cap \{2\gamma\leq x_2\leq 4\gamma\}$, by Fubini's theorem 
and \eqref{small_volume_and_perimeter_of_minority} (for $\rho=1$), we have 
\begin{align}\label{upper_horizontal_Fubini_1}
\begin{split}
\int_{2\gamma}^{4\gamma}\Big[\int_{\Gamma^h_s}\chi \,\mathrm{d}\H^1+|D\chi|\big(\Gamma^h_s\big)\Big]\,\mathrm{d}s&\leq \|\chi\|_{L^1(B^+_{1,\gamma}(\vec{z}_2))}+|D\chi|(B^+_{1,\gamma}(\vec{z}_2))\\
&\leq\|\chi\|_{L^1(B_1^+(\vec{z}_2))}+|D\chi|({B_1^+(\vec{z}_2)})\leq 2 c\,,
\end{split}
\end{align}
where the constant $c>0$ will be chosen sufficiently small. In particular, if we set 
\begin{equation}\label{eq: good_set_of_slices}
I_c:=\Big\{s\in [2\gamma,4\gamma] \colon \int_{\Gamma^h_s}\chi \,\mathrm{d}\H^1+|D\chi|\big(\Gamma^h_s\big)\leq \frac{2 c}{\gamma} \Big\}\,,	
\end{equation}	 
\eqref{upper_horizontal_Fubini_1} implies that 
\begin{equation}\label{eq:length_of_good_set_of_slices}
\L^1(I_c)\geq \gamma\,.
\end{equation}
Choosing \[0<c<\frac{\gamma}{4 }=\frac{\nu_1}{200}\,,\]
and by the general slicing properties of $BV$ functions, we have that $\chi|_{\Gamma_s^h}\in BV(\Gamma_s^h;\{0,1\})$ for $\L^1$-a.e. $s\in I_c$, and  
\begin{equation}\label{smallness_in_horizontal_slice}
\int_{\Gamma_s^h} \chi\,\mathrm{d}\H^{1}+|D\chi|({\Gamma_s^h})< \frac{1}{2}\,.
\end{equation} 
In particular, since $|D(\chi|_{\Gamma_s^h})|(\Gamma_s^h)\in \N$,  \eqref{smallness_in_horizontal_slice} implies that $D(\chi|_{\Gamma_s^h})\equiv 0$, \textit{i.e.}, $\chi|_{\Gamma_s^h}$ is constant $\H^1$-a.e. in $\Gamma_s^h$, and since it takes values in $\{0,1\}$, \eqref{eq:length_of_Gamma} and \eqref{smallness_in_horizontal_slice} in turn imply that 
\[\chi\equiv 0\ \H^1\text{-a.e. on } \Gamma_s^h, \text{ for } \L^1\text{-a.e. } s\in I_c\,.\] 
Using Fubini's theorem and \eqref{eq: basic_contradiction_inequality_1}, we obtain
\begin{align*}\label{eq: 2nd_Fubini}
\int_{I_c}\Big(\int_{\Gamma_s^h}|\nabla u|^2\,\mathrm{d}\H^1\big)\,\mathrm{d}s=\int_{I_c}\Big(\int_{\Gamma_s^h}|\nabla u-\chi G|^2\,\mathrm{d}\H^1\big)\,\mathrm{d}s\leq \int_{B_1^+(\vec{z}_2)} |\nabla u-\chi G|^2\,\mathrm{d}x< c|G|^2\mu_\gamma^2\,,
\end{align*} 
and thus, recalling \eqref{eq:length_of_good_set_of_slices}, there exists $\sigma\in I_c\subset[2\gamma,4\gamma]$ such that $\chi\equiv 0 \ \H^1$-a.e. on $\Gamma_\sigma^h$, and  
\begin{equation*}
\int_{\Gamma_\sigma^h}|\nabla u|^2\,\mathrm{d}\H^1\leq \frac{c}{\gamma}|G|^2\mu_\gamma^2\,,
\end{equation*}
which proves \eqref{properties_of_tilting}(iii),(iv). 
\typeout{References}

\end{document}